\documentclass[11pt,reqno]{amsart}

\date{August 11, 2012}

\usepackage[ansinew]{inputenc}
\usepackage{textcomp}
\usepackage{cite}

\usepackage{amsthm}
\usepackage{amssymb}
\usepackage{amsfonts}

\newtheorem{theorem}{Theorem}[section]
\newtheorem{proposition}[theorem]{Proposition}
\newtheorem{lemma}[theorem]{Lemma}
\newtheorem{corollary}[theorem]{Corollary}

\theoremstyle{definition}

\theoremstyle{remark}
\newtheorem*{remark}{Remark}

\numberwithin{equation}{section}

\newcommand{\da}{\downarrow}
\newcommand{\lo}{L^{(1)}_d}
\newcommand{\lk}{\left(}
\newcommand{\lz}{L^{(2)}_d(b)}
\newcommand{\N}{\mathbb{N}}
\newcommand{\R}{\mathbb{R}}
\newcommand{\rk}{\right)}
\newcommand{\Sph}{\mathbb{S}}
\newcommand{\tr}{\textnormal{Tr}}
\newcommand{\Z}{\mathbb{Z}}

\DeclareMathOperator{\dist}{dist}

\begin{document}

\title[Semi-classics with boundary conditions]{Semi-classical analysis of the Laplace operator\\ with Robin boundary conditions}

\author[R. L. Frank]{Rupert L. Frank}
\address{Rupert L. Frank, Department of Mathematics, Princeton University, Princeton, NJ 08544, USA}
\email{rlfrank@math.princeton.edu}

\author[L. Geisinger]{Leander Geisinger}
 \address{Leander Geisinger, Department of Mathematics, Princeton University, Princeton, NJ 08544, USA} \email{leander@princeton.edu}

\thanks{\copyright\, 2012 by the authors. This paper may be reproduced, in its entirety, for non-commercial purposes.}

\thanks{The authors wish to thank A. Laptev for stimulating their interest in this problem. U.S. NSF grants PHY-1068285 (R.F.) and PHY-1122309 (L.G.) and DFG grant GE 2369/1-1 (L.G.) are acknowledged.}

\begin{abstract}
We prove a two-term asymptotic expansion of eigenvalue sums of the Laplacian on a bounded domain with Neumann, or more generally, Robin boundary conditions. We formulate and prove the asymptotics in terms of semi-classical analysis. In this reformulation it is natural to allow the function describing the boundary conditions to depend on the semi-classical parameter and we identify and analyze three different regimes for this dependence.
\end{abstract}


\maketitle

\section{Introduction and main result}
\label{}


\subsection{Introduction}
The Laplace operator on a bounded domain $\Omega \subset \R^d$, $d \geq 2$, initially defined as a symmetric operator in $L^2(\Omega)$ with domain $C_0^\infty(\Omega)$, admits various  self-adjoint extensions that correspond to different boundary conditions. Our goal in this paper is to study how different boundary conditions influence the asymptotic behavior of the eigenvalues.

We consider self-adjoint extensions that are generated by a quadratic form
\begin{equation}
\label{eq:basicform}
\int_\Omega |\nabla v|^2 dx + \int_{\partial \Omega} c(x) |v(x)|^2 d\sigma(x) \, , \quad v \in H^1(\Omega) \, .
\end{equation}
Here the form domain $H^1(\Omega)$ is the Sobolev space of order $1$, $d\sigma$ denotes the $d-1$-dimensional surface measure on the boundary $\partial \Omega$, and $c$ is a bounded, real valued function on $\partial \Omega$. This quadratic form induces a unique self-adjoint operator $-\Delta_c$ in $L^2(\Omega)$ and functions from the domain of $-\Delta_c$ satisfy, in an appropriate sense, Robin boundary conditions
\begin{equation}
\label{eq:bcint}
\frac{\partial v}{\partial n_x} (x) =  c(x) v(x) \, , \quad x \in \partial \Omega \, ,
\end{equation}
where $\frac{\partial}{\partial n_x}$ denotes the inner normal derivative.  We remark that $c \equiv 0$ corresponds to the important case of Neumann boundary conditions. The Dirichlet Laplacian, generated by the quadratic form $\int_\Omega |\nabla v|^2 dx$ with form domain $H_0^1(\Omega)$, can be recovered formally by taking the limit $c \to \infty$.

If the boundary of $\Omega$ is sufficiently regulary (e.g., Lipschitz continuous), the spectrum of $-\Delta_c$ is purely discrete: It consists of a sequence of eigenvalues $\lambda_1 < \lambda_2 \leq \lambda_3 \leq \dots$ that accumulate at infinity only. Here we study how the asymptotic distribution of the eigenvalues depends on the boundary condition induced by the function $c$.

It is a classical result that the eigenvalues satisfy
\begin{equation}
 \label{eq:weyl}
\lambda_n = \frac{4\pi^2}{(\omega_d |\Omega|)^{2/d}} \, n^{2/d} + o(n^{2/d}) \quad \textnormal{as} \quad n \to \infty \, ,
\end{equation}
where $|\Omega|$ is the volume of $\Omega$ and $\omega_d$ denotes the volume of the unit ball in $\R^d$. In the case of Dirichlet boundary conditions these asymptotics go back to \cite{Wey12a}. They have been generalized in various ways, in particular, to the case of Robin boundary conditions \eqref{eq:bcint}; see, for instance, the lecture notes \cite{BirSol80}.

It has been conjectured by Weyl that \eqref{eq:weyl} is the beginning of an asymptotic expansion in $n$ and that the second term should depend on the surface area of $\Omega$. Initially, a weaker form of this conjecture has been verified, not for individual eigenvalues, but for smooth functions of the eigenvalues; see, e.g., \cite{Ple54,McKSin67}. For instance, \cite{BrGi90} computed in the case of boundary conditions \eqref{eq:bcint}
\begin{align}
 \label{eq:heat}
\sum_{j=1}^\infty e^{-t\lambda_j} = (4\pi t)^{-d/2} & \left( |\Omega|  + \frac{\sqrt\pi}2 |\partial\Omega|\ t^{1/2} + \frac13 \int_{\partial\Omega} \left(H(x) - 6 c(x)\right) d\sigma(x) \ t + O(t^{3/2}) \right) \notag \\
 & \textnormal{as} \quad t \to 0 \, .
\end{align}
Here $H(x)$ is the mean curvature (the trace of the second fundamental form) at $x\in\partial\Omega$. We see that the second term indeed depends on the surface area $|\partial \Omega|$ and is independent of $c$. The boundary condition enters only in the third order term. (For Dirichlet conditions, however, the sign of the second term flips.) In contrast to \eqref{eq:weyl}, the expansion \eqref{eq:heat} requires the boundary to be smooth.

A two-term asymptotic formula for individual eigenvalues was eventually shown in a celebrated work of V. Ivrii; see \cite{Ivr80a,Ivr80b,SafVas97,Ivr98}. He showed that, under a certain condition on the global geometry of $\Omega$ (and some smoothness conditions), one has for boundary conditions \eqref{eq:bcint}
\begin{equation}
\label{eq:evasympt}
\lambda_n = \frac{4\pi^2}{(\omega_d |\Omega|)^{2/d}} \, n^{2/d} - \frac {2\pi^2} d \frac{\omega_{d-1} |\partial \Omega|}{(\omega_d |\Omega|)^{1+1/d}} \, n^{1/d}+ o(n^{1/d}) \quad \textnormal{as} \quad n \to \infty \,.
\end{equation}
Again, for any bounded function $c$ the result is the same as for Neumann conditions. We emphasize that \eqref{eq:evasympt} implies the two-term analogue of \eqref{eq:heat}, but not vice versa.

In this paper we shall study an eigenvalue quantity which is intermediate between \eqref{eq:heat} and \eqref{eq:evasympt}, namely, partial sums $\sum_{j=1}^n \lambda_j$ as $n\to\infty$ or, equivalently, $\sum_{j=1}^\infty (\lambda_j - \mu)_-$ as $\mu\to\infty$. These partial sums describe the energy of non-interacting fermionic particles in $\Omega$ at fixed particle number $n$ or at fixed chemical potential $\mu$, respectively. They play an important role in physical applications.

Since the function $\lambda\mapsto(\lambda-\mu)_-$ is not smooth, we cannot expect that a three-term asymptotic expansion exists for these eigenvalue sums. Hence, to see the effect of boundary conditions already in the second term of the asymptotic expansion we have to choose energy-dependent boundary conditions. Let us state this problem in a semi-classical set-up. For a small parameter $h >0$ we define self-adjoint operators $H(b) = -h^2 \Delta_{b/h} - 1$ in $L^2(\Omega)$ generated by the quadratic form
\begin{equation}
\label{eq:formint}
q_b[v] = h^2 \int_\Omega |\nabla v|^2 dx + h \int_{\partial \Omega}
b(x) |v(x)|^2 d\sigma(x) - \int_{\Omega} |v(x)|^2 dx
\end{equation}
with form domain $H^1(\Omega)$. Here $b$ is a bounded function on $\partial \Omega$ that may also depend on $h$.  The quadratic form $q_b$  induces, in an appropriate sense, $h$-dependent boundary conditions
\begin{equation}
\label{eq:robin}
h \frac{\partial v}{\partial n_x} (x) =  b(x) v(x) \, , \quad x \in \partial \Omega \, .
\end{equation}
In this introduction, we denote by $E_n(b,h)$ the eigenvalues of the operator $-h^2\Delta_{b/h}$; consequently, the eigenvalues of $H(b)$ are given by $E_n(b,h) - 1$. As we explained, our main goal will be to study the sum of the negative eigenvalues of $H(b)$,
$$
\tr H(b)_- = \sum_{n\in \N} (E_n(b,h) - 1 )_- \, ,
$$
in the semiclassical limit $h \da 0$. We prove two-term asymptotics and show how the second term depends on the function $b$. Our analysis will show that the asymptotics has different forms in \emph{three different regimes} depending on the size of $b$ as $h\da 0$. The three different regimes are where $b\to 0$ as $h\da 0$, $b$ of order one as $h\da 0$ and $|b|\to\infty$ as $h\da 0$.

As an example of the first regime, let us consider the case $b = h c$ with a bounded function $c$ independent of $h$. This corresponds to the classical situation discussed above, where the boundary condition \eqref{eq:robin} is independent of $h$ and therefore the eigenvalues $E_n(b,h)=h^2 \lambda_n$ depend trivially on $h$. Then \eqref{eq:evasympt} implies
\begin{equation}
\label{eq:evasymptsum}
\frac1n\sum_{j=1}^n \lambda_j = \frac{4\pi^2}{(\omega_d |\Omega|)^{2/d}} \frac{d}{d+2} \, n^{2/d} - \frac {2\pi^2} d \frac{\omega_{d-1} |\partial \Omega|}{(\omega_d |\Omega|)^{1+1/d}} \frac{d}{d+1} \, n^{1/d}+ o(n^{1/d}) \quad \textnormal{as} \quad n \to \infty \, ,
\end{equation}
and this is equivalent, by a simple majorization argument, to
\begin{equation}
\label{eq:neumann}
\tr H(b)_- = L_d^{(1)} |\Omega| h^{-d} + \frac 14 L_{d-1}^{(1)} |\partial \Omega| h^{-d+1} + o(h^{-d+1}) \quad \textnormal{as} \quad h \da 0
\end{equation}
with $L_d^{(1)} = \frac{2}{d+2} (2\pi)^{-d} \omega_d$. Of course, we find again that the first two terms of the asymptotics are independent of the boundary condition. As we shall see, this is characteristic for the whole regime where $b\to 0$ as $h\da 0$. We emphasize that as a byproduct of our analysis we establish \eqref{eq:neumann} independently, without using \eqref{eq:evasympt}; see Theorem \ref{thm:small}. This includes, as a special case, the Neumann Laplacian.

Among the three regimes mentioned above, the technically most interesting one is when $b$ is independent of $h$. In this case the second term of the semi-classical limit of $\tr H(b)_-$ does depend on the local behavior of $b(x)$; see Theorem \ref{thm:1} below.

Finally, in Theorem \ref{thm:large}, we consider functions $b$ such that $|b|$ diverges as $h \da 0$. In this case, the form of the asymptotics depends on whether $b$ is negative somewhere or whether $b$ is non-negative. In the first case, the asymptotics are determined by the negative part of $b$ alone. Moreover, if $b$ diverges fast enough, then the boundary term becomes the leading term and diverges faster than the Weyl term. On the other hand, when $b$ is non-negative the order of the second term is preserved but the coefficient may change.

We obtain these results by further extending the approach developed in \cite{FraGei11,FraGei12}, where we treated the Dirichlet Laplacian and the fractional Laplacian on a domain. One virtue of this approach is that it requires only rather weak regularity assumptions on $\partial\Omega$ and $b$. Essentially, a $C^1$ assumption on $\partial\Omega$ and on $b$ suffices for a two-term asymptotics.

We now turn to a more precise description of our assumptions and results.


\subsection{Main Results}
Let $\Omega\subset\R^d$, $d\geq 2$, be a bounded domain such that the boundary satisfies a uniform $C^1$ condition. That is, the local charts of $\partial\Omega$ are differentiable and their derivatives are uniformly continuous and share a common modulus of continuity; see (\ref{eq:red:fest}) for a precise definition.
Moreover, we assume that the boundary coefficient $b$ is a continuous, real-valued function on $\partial\Omega$ and we denote a modulus of continuity by $\beta$, i.e.,
\begin{equation}
\label{eq:hoelder}
\left| b(x) - b(y) \right| \leq \beta( |x-y|)
\end{equation}
for all $x,y \in \partial \Omega$. We assume that $\beta$ is non-decreasing.

We remark that the boundary conditions \eqref{eq:robin} for functions in the operator domain of $H(b)$ need not hold in the classical sense under these weak assumptions on the boundary. For $\partial \Omega \in C^1$, however, this operator can still be defined by means of the quadratic form $q_b$ and the characterization of the operator domain in terms of the form domain gives a weak sense in which \eqref{eq:robin} are valid. This suffices for our proof.

For a constant $b \in \R$ we set 
\begin{align}
\label{eq:lz}
 \lz = \begin{cases} C_d  \lk - \frac \pi 4 + \int_0^1 (1-p^2)^{(d+1)/2} \frac{b}{b^2+p^2} dp \rk & \textnormal{for} \ b > 0 \, ,\\
 C_d \,\frac \pi 4 &  \textnormal{for} \ b = 0 \, ,  \\
         C_d  \lk - \frac \pi 4 + \int_0^1 (1-p^2)^{(d+1)/2} \frac{b}{b^2+p^2} dp  +  \pi (b^2+1)^{(d+1)/2} \rk & \textnormal{for} \ b < 0 \, ,
       \end{cases}
\end{align}
where $C_d = 4 |\Sph^{d-2}| (2\pi)^{-d} (d^2-1)^{-1}$.
This expression comes from the explicit diagonalization of a one-dimensional model operator; see Section \ref{sec:half}. Although it is not obvious from the definition, the function $\lz$ is continuously differentiable and non-increasing; see Lemma~\ref{lem:lz} and the remark after Proposition \ref{pro:half}. In particular, for $b > 0$, we have
\begin{equation}
\label{eq:continuity}
-\frac 14 L^{(1)}_{d-1} \, = \, \lim_{b \to \infty} L^{(2)}_d(b)  \leq \, \lz \, \leq \, \lim_{b \da 0} L^{(2)}_d(b) \,= \, L^{(2)}_d(0) \, = \, \frac 14 L^{(1)}_{d-1}
\end{equation}
with $L^{(1)}_{d-1}$ defined after \eqref{eq:neumann}.

To control error terms we have to introduce a non-decreasing function $\delta: [0,\|b\|_\infty]\to [0,\infty)$ such that
\begin{equation}
\label{eq:delta}
\delta(\lambda) \geq \left| \{ x \in \partial \Omega \, : \, 0<|b(x)| < \lambda \} \right|
\end{equation}
for all $0<\lambda \leq \|b\|_\infty$.

Our first main result is the following.

\begin{theorem}
\label{thm:1}
Let $\partial \Omega \in C^1$ and assume that $b$ satisfies \eqref{eq:hoelder} and \eqref{eq:delta} with $\beta(l)=o(1)$ and $\delta(l)=o(1)$ as $l \da 0$.
We write
$$
\tr (H(b))_- \, = \, \lo \, |\Omega| \, h^{-d} + \int_{\partial \Omega} L^{(2)}_d(b(x)) d\sigma(x) \, h^{-d+1} + R_h \, .
$$
Then, for an $h$-independent domain $\Omega$, a given $h$-independent upper bound on $\|b\|_\infty$ and given $h$-independent $\beta$ and $\delta$, the asymptotics 
$$
R_h = o(h^{-d+1})
$$
holds uniformly in $b$ satisfying these conditions. 
\end{theorem}

In other words, in this theorem we claim that $R_h = o(h^{-d+1})$ if $b$ is independent of $h$. 
Moreover, we claim that these asymptotics are valid even if $b$ depends on $h$, as long as it can be controlled in some uniform way. More precisely, we prove that given $\beta$ and $\delta$ (both non-decreasing and vanishing at zero) and constants  $C>0$ and $\varepsilon > 0$, there is an $h_\varepsilon > 0$ such that $|R_h| \leq \varepsilon h^{-d+1}$ for all $0 < h \leq h_\varepsilon$ and all $b$ satisfying $\|b\|_\infty \leq C$, \eqref{eq:hoelder} and \eqref{eq:delta}. Our proof would also allow us to consider $h$-dependent domains $\Omega$, but we do not track the dependence of the constants in terms of $\Omega$ for the sake of simplicity.

Our next result concerns the case where $\|b\|_\infty \to 0$ as $h\da 0$. We will see that the asymptotics are the same as in Theorem \ref{thm:1} with $b=0$. We cannot apply Theorem \ref{thm:1}, however, since for $b\not\equiv 0$ we cannot choose $\delta$ independent of $h$ such that \eqref{eq:delta} is satisfied and $\delta(\lambda)=o(1)$ as $\lambda \da 0$. Moreover, we can dispense with the assumption that $b$ is continuous.

\begin{theorem}
\label{thm:small}
Let $\partial \Omega \in C^1$ and assume that $b = \theta(h) b_0$  with $\theta(h) = o(1)$ as $h \da 0$ and with a bounded function $b_0$. We write
$$
\tr (H(b))_- \, = \, \lo \, |\Omega| \, h^{-d} + \frac 14 L^{(1)}_{d-1} |\partial \Omega| \, h^{-d+1} + R_h \, .
$$
Then, for an $h$-independent domain $\Omega$ and a given $h$-independent upper bound on $\|b_0\|_\infty$, the asymptotics
$$
R_h = o(h^{-d+1})
$$
holds uniformly in $b$ satisfying these conditions.
\end{theorem}

We refer to \eqref{eq:rem2} for an explicit bound on $R_h$.

Our third result concerns the case where $b = \Theta(h) b_0$ with $\Theta(h)\to\infty$.

\begin{theorem}
\label{thm:large}
Let $\partial \Omega \in C^1$. Assume $b = \Theta(h) b_0$ with $\Theta^{-1}(h) = o(1)$ as $h \da 0$ and with $b_0$ satisfying \eqref{eq:hoelder} with $\beta(l)=o(1)$ as $l\da 0$. We write
$$
\tr (H(b))_- \, = \, \lo \, |\Omega| \, h^{-d} + \pi C_d \int_{\partial \Omega} b(x)_-^{d+1}  d\sigma(x) \, h^{-d+1} + R_h \, .
$$
Then, for an $h$-independent domain $\Omega$, a given $h$-independent upper bound on $\|b_0\|_\infty$ and a given $h$-independent $\beta$, the asymptotics
$$
R_h = o(\Theta(h)^{d+1}h^{-d+1})
$$
holds uniform in $b$ satisfying these conditions. 

If, in addition, $b(x) \geq 0$ for all $x \in \partial \Omega$, $\delta(\lambda) = o(1)$ as $\lambda \da 0$, and $\beta(Mh)\Theta(h) = o(1)$ as $h \da 0$ for every fixed $M>0$, then
$$
\tr (H(b))_- \, = \, \lo \, |\Omega| \, h^{-d} - \frac 14 L^{(1)}_{d-1} \, |\partial \Omega_+| \, h^{-d+1} + \frac 14 L^{(1)}_{d-1} \, |\partial \Omega_0| \, h^{-d+1} + o(h^{-d+1}) \, ,
$$
where $\partial \Omega_+ = \{x \in \partial \Omega \, : \, b(x) > 0 \}$ and $\partial \Omega_0 = \{x \in \partial \Omega \, : \, b(x) = 0 \}$. 
\end{theorem}

We emphasize that, if the negative part of $b$ does not vanish and $\Theta(h) = h^\gamma$ with $\gamma = 1/(d+1)$, then the order of the boundary term is the same as the order of the Weyl term. For $\gamma > 1/(d+1)$ the boundary term becomes the leading term.

Since $\beta(l)$ vanishes at most linearly in $l$ for non-constant $b$, the condition $\beta(Mh)\Theta(h) = o(1)$ as $h \da 0$ in the second part of the theorem implies $\Theta(h)=o(h^{-1})$. Our techniques do not allow us to consider faster growing $b$'s and we do not know whether one still can expect the result in that case.


\section{Strategy of the proof}
\label{sec:proof}

In this section we outline the main steps of our proof. In particular, we explain how the main results follow from local estimates. 

First, we localize the operator $H(b)$ into balls, whose size varies depending on the distance to the complement of $\Omega$ \cite{Hoe85,SolSpi03}. Then we analyze the local asymptotics separately in the bulk and close to the boundary.

To localize, let $d(u) = \inf \{|x-u| \, : \,  x \notin \Omega  \}$ denote the distance of $u \in \R^d$ to the complement of $\Omega$. We set 
\begin{equation}
\label{eq:ludef}
l(u) \, = \, \frac 12 \lk 1 + \lk d(u)^2 + l_0^2 \rk^{-1/2} \rk^{-1} \, , 
\end{equation}
where  $0 < l_0 \leq 1$ is a parameter depending only on $h$. Eventually, we will choose $l_0 = o(1)$ as $h \da 0$.
In Section \ref{sec:loc:s1} we introduce real-valued functions $\phi_u \in C_0^\infty(\R^d)$ with support in $B_u = \{ x \in \R^d \, : \, |x-u| < l(u) \}$. For all $u \in \R^d$ these functions satisfy
\begin{equation}
\label{eq:int:grad:s1}
\left\| \phi_u \right\|_\infty \, \leq \, C \ ,  \qquad \left\| \nabla \phi_u \right\|_\infty \leq C \, l(u)^{-1}
\end{equation}
and, for all $x \in \R^d$,
\begin{equation}
\label{eq:int:unity:s1}
\int_{\R^d} \phi_u^2(x) \, l(u)^{-d} \, du \, = \, 1 \, .
\end{equation}
Here and in the following the letter $C$ denotes various positive constants that  are independent of $u$, $l_0$ and $h$, but may vary from line to line.
To estimate error terms in the following results we put 
$$
b_m = \inf_{x \in \partial \Omega} b(x) \, .
$$

\begin{proposition}
\label{pro:loc:s1}
There is a constant $C_\Omega > 0$ such that for $0 <  l_0 \leq C_\Omega^{-1}$ and $0 <  h \leq l_0/4$ the estimates
\begin{align*}
& - C \lk 1+(b_m)_-^{d+1} h l_0^{-1} \rk l_0^{-1}  h^{-d+2}  \leq \int_{\R^d} \textnormal{Tr} \lk \phi_u H(b) \phi_u \rk_- l(u)^{-d} \, du  -  \textnormal{Tr} (H(b))_-   \leq 0 
\end{align*}
hold.
\end{proposition}

This proposition will be proved in Section \ref{sec:loc:s1}.

In view of this result one can analyze the asymptotic behavior of  $\tr(\phi_u H(b) \phi_u)_-$ separately on different parts of $\Omega$. First, we consider the bulk, where the influence of the boundary is not felt.

\begin{proposition}
\label{pro:bulk}
Let $\phi \in C_0^1(\Omega)$ be supported in a ball of radius $l > 0$ and let 
\begin{equation}
\label{eq:int:gradphi}
\| \nabla \phi \|_\infty \, \leq \, C_\phi \, l^{-1} \, .
\end{equation}
Then for all $h > 0$ the estimates
\begin{equation}
\label{eq:bk}
0 \leq  \lo \int_\Omega \phi^2(x)  dx \, h^{-d}  - \tr \lk \phi H(b) \phi \rk_- \leq  C l^{d-2} h^{-d+2}
\end{equation}
hold, with a constant $C > 0$ depending only on $C_\phi$.
\end{proposition}

For $\phi \in C_0^1(\Omega)$ we have $\phi H(b) \phi$ = $\phi (-h^2 \Delta -1 ) \phi$, where $-\Delta$ is defined on the whole space $L^2(\R^d)$ with form domain $H^1(\R^d)$. Hence, this result is independent of the boundary coefficient $b$ and the proof of Proposition \ref{pro:bulk} is the same as in \cite{FraGei11}.

Close to the boundary of $\Omega$, more precisely, if the support of $\phi$ intersects the boundary, a term of order $h^{-d+1}$ appears that depends on $b$. In this situation let $B$ be a ball containing the support of $\phi$ and put
\begin{equation}
\label{eq:bpm}
b^- = \inf_{x \in \partial \Omega \cap B} b(x) \, , \quad b^i = \inf_{x \in \partial \Omega \cap B} |b(x)| \, , \quad b^s = \sup_{x \in \partial \Omega \cap B} |b(x)| \, .
\end{equation}
To state the remainder estimate we denote by $\omega$ a modulus of continuity of the boundary of $\Omega$; see \eqref{eq:red:fest} for a precise definition.

\begin{proposition}
\label{pro:boundary}
Let $\phi \in C_0^1(\R^d)$ be supported in a ball of radius $l > 0$ and let inequalities \eqref{eq:hoelder} and \eqref{eq:int:gradphi} be satisfied. Then there is a constant $C_\Omega > 0$ such that for $0< l \leq C_\Omega^{-1}$ and $0 < h \leq l$  we have
\begin{equation}
\label{eq:bd:1}
\tr \lk \phi H(b) \phi \rk_-  = \lo \int_\Omega \! \phi^2(x)  dx  h^{-d} + \int_{\partial \Omega}  \! L^{(2)}_d(b(x)) \phi^2(x) d\sigma(x)  h^{-d+1} +  R_{bd}(h,l,b^-,b^i)
\end{equation}
with
\begin{align*}
|R_{bd}(h,l,b^-,b^i)|  \leq C \frac{ l^d}{ h^{d}} & \left( \frac{h^2}{l^{2}} \lk 1+\frac{1+(b^-)_-^{d+1} }{b^i} \rk + \omega(l) \left( 1+ \frac hl (b^-)_-^{d+1}\right) \right. \\
 & \quad \left.+\frac{h}{l} \left(1+(b^-)_-^d\right) \beta(l) \right) \, .
\end{align*}
For $b^s \leq h/l$ we also have 
\begin{equation}
\label{eq:bd:2}
\tr \lk \phi H(b) \phi \rk_-  = \lo \int_\Omega \phi^2(x) \, dx \, h^{-d} + \frac 14 L^{(1)}_{d-1} \int_{\partial \Omega}  \phi^2(x) d\sigma(x) \, h^{-d+1} +  R_0(h,l,b^s)
\end{equation}
with 
$$
|R_0(h,l,b^s)| \leq \,   C l^d h^{-d} \lk l^{-2} h^2 +  b^s (1+|\ln b^s|) + \omega(l) \rk \, .
$$
Here the constants $C > 0$ depend only on $\Omega$ and $C_\phi$.
\end{proposition}

The first statement in Proposition \ref{pro:boundary} is the crucial result of this section. It yields a precise estimate with the boundary term including the correct constant $L^{(2)}_d(b)$. However, we obtain an error term that diverges as $b^i \to 0$. To overcome this effect we also need the second statement for $b$ very close to zero. The next lemma is a simplified version of \eqref{eq:bd:2}, where we estimate the boundary term by $Cl^{d-1}h^{-d+1}$.

\begin{lemma}
\label{lem:boundary}
Under the conditions of Proposition \ref{pro:boundary} there is a constant $C_\Omega > 0$ such that for $0< l \leq C_\Omega^{-1}$ and $0 < h \leq l$  we have
\begin{equation}
\label{eq:bd:3}
\tr \lk \phi H(b) \phi \rk_-  = \lo \int_\Omega \phi^2(x) \, dx \, h^{-d} +   R_0'(h,l,b^-)
\end{equation}
with 
$$
|R_0'(h,l,b^-)| \leq \,   C l^d h^{-d} \lk l^{-1} h +  \omega(l) + l^{-1} h (b^-)_-^{d+1} \left(\min\{lh^{-1}(b^-)_-,1\}+\omega(l)\right) \rk \, .
$$
\end{lemma}

Both Proposition \ref{pro:boundary} and Lemma \ref{lem:boundary} will be proved in Section \ref{sec:straight}.

Based on the preceding results we can now give the proofs of our main results. 

\begin{proof}[Proof of Theorem \ref{thm:1}]
We fix two parameters $0 < \lambda \leq 1$ and $0<\mu \leq 1/4$ and set $l_0 = h \mu^{-1}$. Let us recall the definition of $l(u)$ from \eqref{eq:ludef} and of $B_u = \{ x \in \R^d \, : \, |x-u| < l(u) \}$. We set
$$
U = \{ u \in \R^d \, : \, \partial \Omega \cap B_u \neq \emptyset \}  \, .
$$

First, we need to estimate $l(u)$ uniformly. Note that by definition
\begin{equation}
\label{eq:lulower}
l(u) \geq \,  \frac 14 \min \lk d(u), 1 \rk \quad \textnormal{and} \quad l(u) \, \geq \frac{l_0}{4} \geq  h 
\end{equation}
for all $u \in \R^d$. Moreover, for $u \in U$, we have $d(u) \leq l(u)$ and
\begin{equation}
\label{eq:int:luup:s1} 
l(u) \leq l_0/\sqrt 3 = h / (\sqrt3 \mu) \, .
\end{equation}
For $0 < h \leq  \mu  C_\Omega^{-1}$ it follows that $l_0 \leq C_\Omega^{-1}$ and $l(u) \leq C_\Omega^{-1}$ for all $u \in U$. Moreover, $h=\mu l_0\leq l_0/4\leq l(u)$. Therefore the assumptions of Proposition \ref{pro:loc:s1}, Proposition \ref{pro:bulk}, and Proposition \ref{pro:boundary} are satisfied.

Depending on $\lambda$ we decompose $U$ into the regions
\begin{align*}
U_0 &= \left\{ u \in U \, : \, \exists \, x \in \partial \Omega \cap B_u \, : \, b(x) = 0 \right\} \, , \\
U^* &= \left\{ u \in U \, : \, \forall \, x \in \partial \Omega \cap B_u \, : \, 0 < |b(x)| < \lambda\right\} \, , \\
U_> &=  \left\{ u \in U \, : \, \exists \, x \in \partial \Omega \cap B_u \, : \, |b(x)| \geq \lambda \right\} \, .
\end{align*}
We remark that $U = U_0 \cup U^* \cup U_>$ and that the three sets are mutually disjoint. Indeed, if $x \in \partial \Omega \cap B_u$ with $u \in U_0$, then by the continuity of $b$, see \eqref{eq:hoelder},
\begin{equation}
\label{eq:proofbest}
|b(x)| \leq \beta(l(u)) \leq \beta \lk \frac{h}{\sqrt 3 \mu} \rk \,,
\end{equation}
and similarly, if $x \in \partial \Omega \cap B_u$ with $u \in U_>$,
$$
|b(x)| \geq \lambda - \beta \lk \frac{h}{\sqrt 3 \mu} \rk \,.
$$
Thus, by our assumption on $\beta$, we have for all sufficiently small $h>0$ (depending on $\mu$ and $\lambda$) that $\beta \lk \frac{h}{\sqrt 3 \mu} \rk < \lambda - \beta \lk \frac{h}{\sqrt 3 \mu} \rk$. Thus $U_0\cap U_>=\emptyset$, as claimed. We can also make sure that for all sufficiently small $h$
$$
|b(x)| \leq \sqrt 3 \mu \leq h/l(u)
\qquad
\text{for all}\ x \in \partial \Omega \cap B_u \
\text{with}\ u \in U_0
$$
and
\begin{equation}
\label{eq:blower}
|b(x)| \geq \lambda /2
\qquad
\text{for all}\ x \in \partial \Omega \cap B_u \
\text{with}\ u \in U_> \,.
\end{equation}

To estimate error terms we put, similarly as in \eqref{eq:bpm},
$$
b^-_u = \inf_{x \in \partial \Omega \cap B_u} b(x) \, , \quad b^i_u = \inf_{x \in \partial \Omega \cap B_u} |b(x)| \, , \quad b^s_u = \sup_{x \in \partial \Omega \cap B_u} |b(x)| \, .
$$

First, we apply Proposition \ref{pro:loc:s1}. Then, in order to estimate $\tr (\phi_u H(b) \phi_u )_-$, we use  (\ref{eq:bk}) for $u \in \Omega \setminus U$,  (\ref{eq:bd:1}) for $u \in U_>$, \eqref{eq:bd:2} for $u \in U_0$, and (\ref{eq:bd:3}) for $u \in U^*$. We obtain
$$
- R^- \leq \lo \int_{\R^d} \int_\Omega \phi_u^2(x) \frac{dx \, du}{l(u)^d h^d} +  \int_{U}   \int_{\partial \Omega} L^{(2)}_d(b(x)) \phi_u^2(x)  \frac{d\sigma(x) \,du}{l(u)^d h^{d-1}} - \tr \lk H(b) \rk_- \leq R^+ \, ,
$$
with 
\begin{align*}
R^- = & \int_{U_>} \left|R_{bd}(h,l(u),b^-_u,b^i_u)\right| \frac{du}{l(u)^d} + \int_{U_0} \left| R_0(h,l(u),b^s_u) \right|  \frac{du}{l(u)^d} \\
& +  \int_{U_0}   \int_{\partial \Omega} \left|L^{(2)}_d(b(x)) - \frac 14 L^{(1)}_{d-1} \right| \phi_u^2(x)  \frac{d\sigma(x) \,du}{l(u)^d h^{d-1}} + \int_{U^*} \left|R_0'(h,l(u),b_u^-)\right|   \frac{du}{l(u)^d} \\
& +  \int_{U^*}   \int_{\partial \Omega} \left|L^{(2)}_d(b(x))\right| \phi_u^2(x)  \frac{d\sigma(x) \,du}{l(u)^d h^{d-1}} + C  l_0^{-1}  h^{-d+2} \lk 1+(b_m)^{d+1}_- h l_0^{-1} \rk 
\end{align*}
and 
\begin{align*}
R^+ = & \int_{U_>} \left|R_{bd}(h,l(u),b^-_u,b^i_u)\right| \frac{du}{l(u)^d} + \int_{U_0} \left| R_0(h,l(u),b^s_u) \right|  \frac{du}{l(u)^d} \\
& +  \int_{U_0}   \int_{\partial \Omega} \left|L^{(2)}_d(b(x)) - \frac 14 L^{(1)}_{d-1} \right| \phi_u^2(x)  \frac{d\sigma(x) \,du}{l(u)^d h^{d-1}}+ \int_{U^*}   \left|R_0'(h,l(u),b_u^-)\right|   \frac{du}{l(u)^d} \\
& +  \int_{U^*}   \int_{\partial \Omega} \left|L^{(2)}_d(b(x))\right| \phi_u^2(x)  \frac{d\sigma(x) \,du}{l(u)^d h^{d-1}} + C \int_{\Omega \setminus U} l(u)^{-2} du \, h^{-d+2} \, .
\end{align*}

In the main term we change the order of integration and use the partition of unity property  (\ref{eq:int:unity:s1}) to obtain
$$
\lo \int_{\R^d} \int_\Omega \phi_u^2(x) dx \, \frac{du}{l(u)^d} h^{-d} = \lo |\Omega| h^{-d} 
$$
and
$$
 \int_{U}   \int_{\partial \Omega} L^{(2)}_d(b(x)) \phi_u^2(x)  \frac{d\sigma(x) \,du}{l(u)^d h^{d-1}} = \int_{\partial \Omega}  L^{(2)}_d(b(x)) d\sigma(x) h^{-d+1} \, .
$$
Thus, we get
$$
-R^- \leq \lo |\Omega| h^{-d} +  \int_{\partial \Omega}  L^{(2)}_d(b(x)) d\sigma(x) h^{-d+1}  - \tr \lk H(b) \rk_-   \leq R^+ \, ,
$$
and to complete the proof it remains to bound the remainder terms $R^\pm$.

We now argue that the last term in the definition of $R^+$ is controlled by the last term in the definition of $R^-$, that is, by
\begin{equation}
\label{eq:lastrem}
C l_0^{-1} h^{-d+2} \lk 1+(b_m)_-^{d+1} h l_0^{-1} \rk \leq \, Ch^{-d+1} \mu \lk 1+\|b\|_\infty^{d+1} \rk.
\end{equation}
To prove this, we note that for $u \in \Omega \setminus U$ we have $d(u) \geq l(u) \geq l_0/4$ and 
$$
\int_{\Omega \setminus U} l(u)^{-2} du \, \leq \, C \lk 1 + \int_{\{d(u) \geq l_0/4 \} } d(u)^{-2} du \rk  \leq C \lk 1 + \int_{l_0/4}^\infty t^{-2} \, |\partial \Omega_t | \, dt \rk \, .
$$
Here $|\partial \Omega_t|$ denotes the surface area of the boundary of $\Omega_t = \{ x \in \Omega \, : \, d(x) > t \}$. Using the fact that $|\partial \Omega_t|$ is uniformly bounded and that $|\partial \Omega_t| = 0$ for large $t$, we get
\begin{equation}
\label{eq:int:U1}
\int_{\Omega \setminus U} l(u)^{-2} du  \leq  C l_0^{-1} \leq C \mu h^{-1} \, .
\end{equation}
This proves that the last term in $R^+$ is bounded by \eqref{eq:lastrem}.

To proceed, we note that inequalities (\ref{eq:int:luup:s1}) and (\ref{eq:lulower}) show that $l(u)$  for $u \in U$ is comparable with $l_0$. Since $B_u \cap \partial \Omega \neq \emptyset$ we find $d(u) < l(u) \leq C l_0$ and, for any positive and non-decreasing function $r$,
\begin{equation}
\label{eq:int:U2}
\int_{U} r(l(u)) du \, \leq \, C r(Cl_0) \int_{\{ d(u) \leq l_0 \} } du \, \leq \, C r(Cl_0) l_0\, . 
\end{equation}
Thus, if we insert  the identity $l_0 = h \mu^{-1}$ and the estimates \eqref{eq:proofbest}, \eqref{eq:blower}, (\ref{eq:int:U2}) and  (\ref{eq:int:U1}) into the expressions for $R^-$ and $R^+$, we find that both are bounded by a constant times
\begin{align*}
R = \, &h^{-d+1}  \left(1+ \|b\|_\infty^{d+1} \right) \lk \mu + \frac \mu \lambda+  \omega \lk \frac{Ch}{\mu} \rk \frac 1 \mu + \beta \lk \frac{Ch}{\mu} \rk \rk\\
& + h^{-d+1} \lk |U^*| \frac{\mu}{h} \lk 1+ \omega \lk \frac{Ch}{\mu} \rk \frac 1 \mu  \rk +  \frac 1 \mu \beta \lk \frac{h}{\sqrt 3 \mu} \rk \lk 1 + \left| \ln \beta \lk \frac{h}{\sqrt 3 \mu} \rk \right| \rk \rk \, .
\end{align*}
Here we used the facts that $|U_0| \leq |U| \leq Cl_0$ and $|L^{(2)}_d(b(x)) - \frac 14 L^{(1)}_{d-1}| \leq C \beta(h/\sqrt 3 \mu)$ for $x \in B_u \cap \partial \Omega$ with $u \in U_0$.

To estimate $|U^*|$ we apply Lemma \ref{lem:parallel}, given in the appendix, to the set $N = \{x \in \partial \Omega : 0 < |b(x)|<\lambda \}$. By the defining property \eqref{eq:delta} of $\delta$ we obtain
$$
\limsup_{h \da 0} \frac \mu h |U^*| = \limsup_{l_0 \da 0} \frac 1{l_0} |U^*| \leq C \delta(\lambda) \, .
$$
Hence, by our assumptions on $\omega$ and $\beta$, it follows that
$$
\limsup_{h \da 0} \lk h^{d-1} R \rk \leq \left(1+ \|b\|_\infty^{d+1} \right) \lk \mu + \frac \mu \lambda \rk + C \delta(\lambda) \, .
$$
By our assumption on $\delta$, the right hand side can be made arbitrarily small by choosing first $\lambda$ small and then $\mu$ small. This completes the proof of Theorem \ref{thm:1}.
\end{proof}

\begin{proof}[Proof of Theorem \ref{thm:small}]
This proof is similar to the proof of Theorem \ref{thm:1} above. Again we choose
$$
U =  \{ u \in \R^d \, : \, \partial \Omega \cap B_u \neq \emptyset \} \,. 
$$
and we assume that $l_0 = h \mu^{-1}$ with $0 < \mu \leq 1/4$. Then $h \leq l(u)$ for all $u \in U$.

Let us choose $h$ small enough such that $|b(x)| = |b_0(x)| \theta(h) \leq \sqrt 3 \mu \leq h/l(u)$ for all $x \in \partial \Omega$ and $u \in U$. Then we can apply (\ref{eq:bd:2}) to estimate  $\tr (\phi_u H(b) \phi_u )_-$ for $u \in U$. This yields
$$
\left| \tr(H(b))_- - \lo |\Omega| h^{-d} - \frac 14 L^{(1)}_{d-1} |\partial \Omega| h^{-d+1}  \right| \leq \int_U \left|R_0(h,l(u),b^s_u)\right| \frac{du}{l(u)^d} + Cl_0^{-1} h^{-d+2} \, .
$$
Similarly as above we bound 
$$
\int_U  \left| R_0(h,l(u),b^s_u) \right| \frac{du}{l(u)^d} \leq Ch^{-d+1} \lk \mu + \omega \lk \frac{C h} \mu \rk \frac{1}{\mu} + \|b\|_\infty (1+|\ln \|b\|_\infty |)\frac{1}{\mu} \rk \, .
$$
We multiply this by $h^{d+1}$ and let $h\da 0$ recalling that $\|b\|_\infty =\theta(h)\|b_0\|_\infty = o(1)$. Since $\mu$ can be chosen arbitrarily small, we obtain the claimed asymptotics.
\end{proof}

In this case the proof shows that the remainder $R_h$ from Theorem \ref{thm:small} can be estimated as follows. For all $0 < \mu \leq 1/4$ we have
\begin{equation}
\label{eq:rem2}
|R_h| \leq C h^{-d+1}  \lk \mu + \omega \lk \frac{C h} \mu \rk \frac{1}{\mu} + \theta(h) \|b_0\|_\infty \lk 1+|\ln( \theta(h) \|b_0\|_\infty )| \rk \frac{1}{\mu} \rk \,.
\end{equation}

\begin{proof}[Proof of Theorem \ref{thm:large}]
First, we assume that the negative part of $b$ does not vanish. 
Then in the same way as in  the proof of Theorem \ref{thm:1} we fix parameters $0 < \lambda \leq 1$ and $0 < \mu \leq 1/4$ and set $l_0 = h\mu^{-1}$ and 
$$
U = \{ u \in \R^d \, : \, \partial \Omega \cap B_u \neq \emptyset \}  \, . 
$$
Here we choose
$$
\tilde U^* = \{ u \in U \, : \, \exists \, x \in \partial \Omega \cap B_u \, : \,  |b(x)| < \lambda  \} \, . 
$$

Then, similar as in  the proof of Theorem \ref{thm:1}, by applying (\ref{eq:bk}) for $u \in \Omega \setminus U$,  (\ref{eq:bd:1}) for $u \in U \setminus \tilde U^*$, and \eqref{eq:bd:3} for $u \in \tilde U^*$, we obtain 
$$
\left| \tr(H(b))_- - \lo |\Omega| h^{-d} - \int_{\partial \Omega} L^{(2)}_d(b(x)) dx h^{-d+1} \right| \leq CR \\
$$
with
\begin{align*}
R = & h^{-d+1}  \lk  1+ \|b\|_\infty^{d+1} \rk  \lk  \mu + \frac \mu \lambda + \omega \lk \frac{Ch}\mu \rk \frac 1 \mu + \frac{\Theta(h)}{1+\|b\|_\infty} \beta\lk \frac{Ch}\mu \rk \rk  \\
& + h^{-d+1} \lk 1 + \Theta(h)\beta\lk \frac{Ch}\mu \rk \rk^{d+1} \lk 1 +  \omega \lk \frac{Ch}\mu \rk  \rk \, . 
\end{align*}
We emphasize that in order to arrive at this bound we used the estimates $|\tilde U^*| \leq |U| \leq Cl_0$ and 
$$
|L^{(2)}_d(b(x))| \leq C \lk 1+\Theta(h) \beta\lk \frac{h}{\sqrt 3 \mu} \rk \rk^{d+1}
$$
for $x\in \partial\Omega\cap B_u$ with $u\in U^*$. (Note also that the role of $\beta$ in Proposition \ref{pro:boundary} is now played by $\Theta(h)\beta$.)

To simplify the main term we note that $L^{(2)}_d(b) = C_d \pi b^{d+1} + O(\Theta(h)^{d-1})$ as $h \da 0$. Hence,
$$
\int_{\partial \Omega} L^{(2)}_d(b(x)) dx = C_d \pi \int_{\partial \Omega} b(x)_-^{d+1} d\sigma(x) h^{-d+1} + O(\Theta(h)^{d-1} h^{-d+1}) \, .
$$
It remains to note that 
$$
\limsup_{h \da 0} \lk h^{d-1} \Theta(h)^{-d-1} R \rk \leq C \lk \mu + \frac \mu \lambda  \rk 
$$
can be made arbitrarily small. (Since we only assume an $h$-independent upper bound on $\|b_0\|_\infty$, one needs to distinguish here the cases whether $\liminf \Theta^{-1}(1+\|b\|_\infty)$ is positive or zero.)

We now turn to the proof of the second part of the theorem. If the boundary coefficient $b$ is non-negative we argue in the same way as in the proof of Theorem \ref{thm:1}. We obtain  
\begin{align*}
&\left| \tr(H(b))_- - \lo |\Omega| h^{-d} - \int_{\partial \Omega}  L^{(2)}_d(b(x)) d\sigma(x) h^{-d+1} \right| \\
& \leq Ch^{-d+1}   \lk \mu + \frac \mu \lambda+  \omega \lk \frac{Ch}{\mu} \rk \frac 1 \mu + \Theta(h) \beta \lk \frac{Ch}{\mu} \rk  + |U^*| \frac{\mu}{h} \lk 1+ \omega \lk \frac{Ch}{\mu} \rk \frac 1 \mu  \rk \right. \\
& \quad + \left. \frac 1 \mu \Theta(h) \beta \lk \frac{h}{\sqrt 3 \mu} \rk \lk 1 + \left| \ln \lk \Theta(h) \beta \lk \frac{h}{\sqrt 3 \mu} \rk \rk \right| \rk \rk \, .
\end{align*}
In this case the continuity of $L^{(2)}_d(b)$, see \eqref{eq:continuity}, implies
$$
\int_{\partial \Omega}  L^{(2)}_d(b(x)) d\sigma(x) = - \frac 14 L^{(1)}_{d-1}  |\partial \Omega_+|  + \frac 14 L^{(1)}_{d-1}  |\partial \Omega_0| + o(1) \, ,
$$
by dominated convergence as $h \da 0$. Again applying Lemma \ref{lem:parallel} in the same way as in the proof of Theorem \ref{thm:1} we see that all terms equal $o(h^{-d+1})$ as $h \da 0$.
\end{proof}

To summarize this section, we have reduced the proof of our main results to the proof of Proposition \ref{pro:loc:s1}, Proposition \ref{pro:boundary} and Lemma \ref{lem:boundary}.


\section{Local asymptotics in the half-space}
\label{sec:half}

From a technical point of view, this section is the heart of our proof. We analyze in great detail a model operator which is explicitly diagonalizable. More precisely, we prove local estimates corresponding to Proposition \ref{pro:boundary} in the case where $\Omega$ is the half-space $\R^d_+ = \{ (x',x_d) \in \R^{d-1} \times \R_+ \}$ and the boundary coefficient $b$ does not depend on $x$. Let $H^+(b) = -h^2 \Delta - 1$ be the self-adjoint operator in $L^2(\R^d_+)$ generated by the quadratic form
$$
q_b^+[v] = h^2 \int_{\R^d_+} |\nabla v(x)|^2 dx + h  b \int_{\R^{d-1}} 
 |v(x',0)|^2 dx' - \int_{\R^d_+} |v(x)|^2 dx
$$
with form domain $H^1(\R^d_+)$ and with a real constant $b$ independent of $x$.


\subsection{Statement of the results}\label{sec:half:main}

Our goal in this section is to prove the following

\begin{proposition}
\label{pro:half}
Assume that $b \in \R$ is constant. Let $\phi \in C_0^1(\R^d)$ be supported in a ball of radius $l > 0$ and let \eqref{eq:int:gradphi} be satisfied.
Then for $h > 0$
$$
\tr \lk  \phi H^+(b)  \phi \rk_- = \lo \int_{\R^d_+} \phi^2(x) dx \, h^{-d}  + L^{(2)}_d(b) \int_{\R^{d-1}}  \phi^2(x',0) dx' \, h^{-d+1} + R_{hs}(h,l,b) 
$$
with
$$
|R_{hs}(h,l,b)| \leq C l^{d-2} h^{-d+2} \left( 1+ \frac{1 + b_-^{d+1}}{|b|} \right) \, .
$$
For $|b| \leq  h/l \leq 1$ we also have
$$
\tr \lk  \phi H^+(b)  \phi \rk_- = \lo \int_{\R^d_+} \phi^2(x) dx \, h^{-d}  + \frac 14 L^{(1)}_{d-1} \int_{\R^{d-1}}  \phi^2(x',0) dx' \, h^{-d+1} + R_{hs}'(h,l,b) \, . 
$$
with
$$
|R_{hs}'(h,l,b)| \leq C l^{d-2} h^{-d+2} \lk 1+ l^2 h^{-2} |b| (1+|\ln|b||)  \rk  \, .
$$
Here the constants $C > 0$ depend only on $d$ and $C_\phi$.
\end{proposition} 

\begin{remark}
The proposition shows, in particular, that $L^{(2)}_d(b)$ is non-increasing. Indeed, for given boundary coefficients $b \leq b'$ the variational principle implies $\tr(H(b))_- \geq \tr(H(b'))_-$ for all $h>0$, and Proposition \ref{pro:half} thus yields $L^{(2)}_d(b) \geq L^{(2)}_d(b')$.
\end{remark}

The first part of Proposition \ref{pro:half} is the key semi-classical estimate that we will later generalize to curved boundaries and variable $b$'s. The problem with this bound, however, is the $|b|^{-1}$ in the error term which blows up for small values of $b$. For that reason we need to include the second part, which deals with small values of $b$. (In passing, we note that since $L^{(2)}_d(b)$ is continuously differentiable with $L^{(2)}_d(0)=\frac 14 L^{(1)}_{d-1}$, as we will see in Lemma~\ref{lem:lz}, the constant $\frac 14 L^{(1)}_{d-1}$ in the second part of Proposition \ref{pro:half} can be replaced by $L^{(2)}_d(b)$ without changing the form of the error term.)

To deal with the transition region between $|b|\geq 1$ (where the first part of Proposition \ref{pro:half} applies) and $|b|\leq h/l$ (where the second part applies) we need the following rough estimate.

\begin{lemma}
\label{lem:half}
Assume that $b \in \R$ is constant. Let $\phi \in C_0^1(\R^d)$ be supported in a ball of radius $l > 0$ and let \eqref{eq:int:gradphi} be satisfied.
Then for all $0<h\leq l$ we have
$$
\tr \lk  \phi H^+(b)  \phi \rk_- = \lo \int_{\R^d_+} \phi^2(x) dx \, h^{-d}  + R_{hs}''(h,l,b) \
$$
with
$$
|R''_{hs}(h,l,b)| \leq C l^{d-1} h^{-d+1} \left(1+ b_-^{d+1}  \min\{b_-l h^{-1},1 \}\right) \,.
$$
Here  $C > 0$ depends only on $d$ and $C_\phi$.
\end{lemma}

From this lemma we immediately deduce a simple bound that will be useful in the following sections.

\begin{corollary}
\label{cor:half}
Assume that $b \in \R$ is constant. Let $\phi \in C_0^1(\R^d)$ be supported in a ball of radius $l > 0$ and let \eqref{eq:int:gradphi} be satisfied. Then for all  $0 < h \leq l$ the bound
$$
\tr \lk \phi H^+(b) \phi \rk_- \leq C \, l^d \, h^{-d} \lk 1 + b_-^{d+1} h l^{-1}  \rk
$$
holds with a constant $C$ depending only on $d$ and $C_\phi$.
\end{corollary}

The next remark will be used at several places without explicit mentioning in the proofs of Proposition \ref{pro:half} and Lemma \ref{lem:half}.

\begin{remark}
When bounding error terms in the following proofs we will sometimes encounter the term  $\| \phi \|_\infty$, which is not mentioned in Proposition \ref{pro:half} and elsewhere. The reason is that it can be controlled in terms of $C_\phi$. Indeed, for $x$ in the support of $\phi$ we can choose $y$ at the boundary of the support with $|x-y|\leq l$ and use (\ref{eq:int:gradphi}) to estimate
$$
|\phi(x)| = |\phi(x) - \phi(y)| \leq \|\nabla \phi\|_\infty |x-y| \leq C_\phi \, .
$$
Hence, $\|\phi\|_\infty \leq C_\phi$, as claimed.
\end{remark}


\subsection{Analysis of a model operator on the half-line}

The bounds in Proposition \ref{pro:half} and Lemma \ref{lem:half} are based on the following results about the one dimensional operator $-\frac{d^2}{dt^2}$ on the half-line $\R_+$ with boundary condition
\begin{equation}
\label{eq:bchs}
\partial_t v(0) = b \, v(0) \, , \ b \in \R \, .
\end{equation}
For $t \geq 0$ and $b \in \R$ we define
$$
\psi_b(t) = \frac 1{\sqrt{1+b^2}} \cos(t) + \frac b {\sqrt{1+b^2}}  \sin(t) 
$$
and, for $b < 0$,
$$
\Psi_b(t) = \sqrt{-2b} \, e^{bt} \, .
$$
In order to treat positive and negative $b$ without distinction we set $\Psi_b \equiv 0$ for $b \geq 0$.
Then we have
\begin{align}
\label{eq:genef1}
- \partial^2_t \psi_b(t) &=   \psi_b(t) \, , \\
\label{eq:genef2}
- \partial^2_t \Psi_b(t) &=  -b^2 \, \Psi_b(t) \, , 
\end{align}
and all functions satisfy boundary conditions (\ref{eq:bchs}). These functions form a complete system of (generalized) eigenfunctions: For functions $v \in L^2(\R_+)$ we have
\begin{equation}
\label{eq:genef3}
v(t) \, = \, \int_0^\infty  \lk \frac 2\pi \int_0^\infty  \psi_{b/p}(tp) \, \psi_{b/p}(sp) dp + \Psi_b(t) \Psi_b(s) \rk  v(s) \, ds 
\end{equation}
in the sense of $L^2$-convergence. 
This identity holds for continuous $v \in L^1(\R_+) \cap L^2(\R_+)$  and is extended first to $L^1(\R_+) \cap L^2(\R_+)$ and then to $L^2(\R_+)$ as in the case of the ordinary Fourier transform.

We need the following technical result.

\begin{lemma}
\label{lem:unifbound}
For $t \in \R_+$ and $b \in \R$ we have  
$$
\psi_{b}^2(t) \leq  1 \, .
$$
Moreover, the function
$$
I_b(t) = \int_0^1 (1-p^2)^{(d+1)/2}\lk  \frac{p^2-b^2}{p^2+b^2} \cos(2tp) + \frac{2pb}{p^2+b^2} \sin(2tp) \rk dp 
$$
is uniformly bounded with respect to $t \geq 0$ and $b \in \R$. It satisfies
\begin{equation}
\label{eq:iintbound}
\int_0^\infty |I_b(t)| dt \leq C \quad \mbox{and} \quad \int_0^\infty t |I_b(t)| dt \leq C\times \begin{cases} 1& \mathrm{if}\ b=0 \\ \lk 1+\frac 1{|b|} \rk & \mathrm{if}\ b\neq 0 \end{cases}
\end{equation}
with $C>0$ depending only on the dimension. 
\end{lemma}

\begin{proof}
The first assertion follows directly from the definition of $\psi_b$ since 
\begin{equation}
\label{eq:eigenf}
\psi_{b}^2(t) =  \frac 12 + \frac {(1-b^2) \cos(2t) + 2b \sin(2t)}{2(1+b^2)} = \frac 12 + \frac {(1-ib)^2 e^{i2t} + (1+ib)^2 e^{-i2t}}{4(1+b^2)} \, .
\end{equation}
It is clear from the definition that $I_b$ is uniformly bounded. To establish decay in $t$ we write
$$
I_b(t) = \frac 12 \int_\R (1-p^2)^{(d+1)/2}_+ \frac{(p-ib)^2}{p^2+b^2} e^{i2tp} dp \, ,
$$
and set $G(p) = (1-p^2)_+^{(d+1)/2}$ and $H_b(p) = (p-ib)^2/(p^2+b^2)$. Let $\check G$ and $\check H_b$ denote the inverse (distributional) Fourier transforms of $G$ and $H_b$.

It is well known that $\check G(t) = c_d J_{d/2+1}(|t|) |t|^{-d/2-1}$, 
where $J_{d/2+1}$ denotes the Bessel function of the first kind. The absolute value of this Bessel function behaves like $t^{d/2+1}$ as $t \to 0+$ and is bounded by a constant times $t^{-1/2}$ as $t \to \infty$; see  \cite[(9.1.7) and (9.2.1)]{AbrSte64}.  Hence, we have $|\check G(t)| \leq C \min\{1,|t|^{-(d+3)/2}\}$. Moreover, we compute that
$$
\check H_b(t) = (2\pi)^{1/2}\delta(t) - 2^{3/2} \pi^{1/2} |b|\, \chi_{\R_-}(bt)\, e^{-|bt|} \,.
$$

Thus we may rewrite $I_b(t)$ in terms of $\check G$ and $\check H_b$ and get
\begin{align*}
I_b(t) &= \frac 12 \int_\R \check G(2t-u) \check H_b (u) du \\
&= \left(\frac\pi 2\right)^{1/2} \check G(2t) - (2\pi)^{1/2} |b| \int_\R \check G(2t-u) \chi_{\R_-}(bu)\, e^{-|bu|} du \\
& = \left(\frac\pi 2\right)^{1/2} \check G(2t) - (2\pi)^{1/2} \int_0^\infty \check G \lk 2t+ \frac u b \rk e^{-u} du \, .
\end{align*}
In the last change of variables we have assumed that $b\neq 0$. From the bound $|\check G(t)|=|\hat G(-t)| \leq C \min\{1,|t|^{-(d+3)/2}\}$ we easily derive that $\int_0^\infty |\check G(2t+u/b)| dt \leq C$. Moreover,
\begin{align*}
\int_0^\infty t |\check G(2t+u/b)| dt & =\frac14 \int_{u/b}^\infty \left(t-u/b\right) |\check G(t)| dt \\
& \leq \frac14 \left( \int_\R |t| |\check G(t)| dt + \frac{u}{|b|} \int_\R |\check G(t)| dt \right)
\leq C\left( 1+ \frac u{|b|}\right) \,.
\end{align*}
This implies \eqref{eq:iintbound} for $b\neq 0$. The case $b=0$ is similar.
\end{proof}

The next lemma establishes a connection between the function $I_b$ and the coefficient $L^{(2)}_d(b)$ defined in \eqref{eq:lz}.

\begin{lemma}
\label{lem:lz}
For $L^{(2)}_d(b)$ we have the representations
\begin{equation}
\label{eq:lzlem}
 \lz = \begin{cases} C_d  \int_0^\infty I_b(t) dt  & \textnormal{for} \ b \geq 0 \, ,\\
 C_d \lk \int_0^\infty I_b(t) dt + \pi (b^2+1)^{(d+1)/2} \rk  & \textnormal{for} \ b < 0 \, .
       \end{cases}
\end{equation}
The function $b\mapsto\lz$ is countinuously differentiable.
\end{lemma}

\begin{proof}
Because of the first bound in \eqref{eq:iintbound} we may apply the dominated convergence theorem to write 
\begin{align*}
\int_0^\infty I_b(t) dt = \lim_{\epsilon \da 0} & \int_0^1 (1-p^2)^{(d+1)/2} \int_0^\infty e^{-\epsilon t^2} \lk \frac{p^2-b^2}{p^2+b^2} \cos(2tp) dt + \frac{2pb}{p^2+b^2} \sin(2tp) \rk dt  dp \\
 = \lim_{\epsilon \da 0} & \lk \frac{\sqrt \pi}{2} \int_0^{1/ \sqrt \epsilon} (1-\epsilon q^2)^{(d+1)/2} \frac{\epsilon q^2 - b^2}{\epsilon q^2 + b^2} e^{-q^2} dq \right.\\
& \left. + \int_0^1 (1-p^2)^{(d+1)/2} \frac{2pb}{p^2+b^2} \frac{1}{\sqrt \epsilon}F \lk \frac {p}{ \sqrt \epsilon} \rk dp \rk \, ,
\end{align*}
where $F(x) = e^{-x^2} \int_0^x e^{y^2} dy$. Using the fact that
$$
\lim_{\epsilon \da 0} \frac1{\sqrt \epsilon} F \lk \frac p {\sqrt \epsilon} \rk = \frac 1{2p}
$$
we find
$$
\int_0^\infty I_b(t) dt  = - \frac \pi 4 + \int_0^1 (1-p^2)^{(d+1)/2} \frac{b}{b^2+p^2} dp
$$
for $b \neq 0$ and $\int_0^\infty I_b(t) dt = \frac \pi 4$ for $b = 0$.  By \eqref{eq:lz} this yields \eqref{eq:lzlem}.

The fact that $b\mapsto\lz$ is $C^1$ away from $b=0$ is elementary. To prove continuity and differentiability at $b=0$ we again use dominated convergence together with the fact that
$$
\lim_{b \to 0 \pm} \int_0^1 (1-p^2)^{(d+1)/2} \frac{b}{b^2+p^2} dp = \pm \frac \pi 2 \, .
$$
We omit the details.
\end{proof}


\subsection{Proof of Propositions \ref{pro:half} and Lemma \ref{lem:half}}

After these preliminaries we can turn to the proof of local asymptotics on the half-space. We split the proof into three lemmas.

\begin{lemma}
\label{lem:half:basic}
Under the conditions of Proposition \ref{pro:half} we have
\begin{align*}
0 \, \leq  \, & 2C_d \int_{\R^d_+}   \int_0^1 \phi^2(x)  (   1-\xi_d^2 )^{(d+1)/2}   \psi_{b/\xi_d}^2 \lk x_d \xi_d /h \rk d\xi_d dx h^{-d}\\
&  + \pi C_d  (b^2+1)^{(d+1)/2} \int_{\R^d_+}  \phi^2(x) \Psi^2_{b/h}(x_d)  dx   h^{-d+1} - \textnormal{Tr} \lk \phi H^+(b) \phi \rk_- \\
\leq \, & C l^{d-2} h^{-d+2}  ( 1+b_-^{d-1} \min \{ b_-, h/l \} ) \, ,
\end{align*} 
where $C_d$ is given in \eqref{eq:lz}. Here the constant $C > 0$ depends only on $d$ and $C_\phi$.
\end{lemma}

\begin{proof}
First note that we may rescale $\phi$ and thus assume $l = 1$ without changing the value of $b$. Since $b$ is fixed throughout the proof we write $H^+$ instead of $H^+(b)$.

To prove the lower bound we apply the variational principle and obtain
$$
-\tr (\phi H^+ \phi)_- \, = \, \inf_{0 \leq \gamma \leq 1} \tr (\gamma \phi H^+ \phi) \geq   \inf_{0 \leq \gamma \leq 1} \lk - \tr (\gamma \phi (H^+)_- \phi) \rk \, = \, - \tr (\phi (H^+)_- \phi ) \, .
$$
Let $a^+(x,y)$ denote the integral kernel of $(H^+)_-$. From \eqref{eq:genef1}, \eqref{eq:genef2}, and \eqref{eq:genef3} we see that
\begin{align*}
a^+(x,y) = \, & \frac{4}{(2\pi h)^d} \int_{\R^d_+} \lk |\xi|^2-1 \rk_- e^{i \xi' \cdot (x'-y')/h}  \psi_{b/\xi_d}(x_d \xi_d/h) \psi_{b/\xi_d}(y_d \xi_d/h) d\xi \\
& + \frac{1}{(2\pi h)^{d-1}} \int_{\R^{d-1}} \lk |\xi'|^2-b^2-1 \rk_-  e^{i \xi' \cdot (x'-y')/h}  \Psi_{b/h}(x_d) \Psi_{b/h}(y_d) d\xi'
\end{align*}
and we get 
\begin{align*}
\tr \lk \phi H^+ \phi \rk_- \leq \, & \frac{4}{(2\pi h)^d} \int_{\R^d_+} \int_{\R^d_+} \phi^2(x) (|\xi|^2-1)_- \psi_{b/\xi_d}^2 \lk x_d \xi_d /h  \rk d\xi dx \\
& + \frac{1}{(2\pi h)^{d-1}}\int_{\R^d_+} \int_{\R^{d-1}} \phi^2(x) (|\xi'|^2-b^2-1)_-  \Psi_{b/h}^2\lk x_d \rk d\xi' dx \, .
\end{align*}
Here we perform the $\xi'$-integration and obtain the lower bound.

We proceed to prove the upper bound. To simplify notation write 
\begin{align*}
f(x,\xi) &= e^{ix' \cdot \xi'} \psi_{b/(\xi_d h)}(x_d \xi_d) \\
F(x,\xi') &= e^{ix' \cdot \xi'} \Psi_{b/h}(x_d) \, .
\end{align*}
We define the operator $\gamma =  (H^+)_-^0 $ with kernel $\gamma(x,y) = \gamma_1(x,y) + \gamma_2(x,y)$, where
\begin{align*}
\gamma_1(x,y) \, &= \, \frac{4}{(2 \pi h)^d}  \int_{ \{\xi \in \R^d_+ \, : \, |\xi|<1 \} }   f\lk x, \xi/h  \rk  \overline{f\lk y, \xi/ h  \rk} \, d\xi  \, , \\
\gamma_2(x,y) \, &= \, \frac{1}{(2 \pi h)^{d-1}} \int_{ \{\xi' \in \R^{d-1} \, : \, |\xi'|^2< b^2+1 \} }   F\lk x, \xi'/h  \rk  \overline{F\lk y, \xi'/ h  \rk} \, d\xi' \, .
\end{align*}
Thus, $\gamma$ satisfies $0 \leq \gamma \leq 1$ and a variant of the variational principle, discussed in Appendix \ref{ap:var}, yields
\begin{equation}
\label{eq:tracesum}
-\tr (\phi H^+ \phi)_- \leq \tr (\phi \gamma \phi H^+ ) = \tr (\phi \gamma_1 \phi H^+ )  + \tr ( \phi \gamma_2 \phi H^+) \, .
\end{equation}
We note that the range of $\phi \gamma \phi$, $\phi \gamma_1 \phi$, and $\phi \gamma_2 \phi$ does not belong to the domain of $H^+$. However, the functions $\phi f$ and $\phi F$ belong to the form domain $H^1(\R^d_+)$ of $H^+$. Therefore \eqref{eq:tracesum} is valid if we interpret $\tr (\phi \gamma \phi H^+ )$ in the sense described in the appendix, namely
\begin{equation}
\label{eq:halfbasic}
\tr (\phi \gamma_1 \phi H^+) =  \frac{4}{(2\pi h )^d} \int_{ \{\xi \in \R^d_+ \, : \, |\xi|<1 \} }  q_b^+[\phi f ] \,  d\xi \, ,
\end{equation}
where
$$
q_b^+[\phi f ] = h^2 \left\| \nabla (\phi f) \right\|^2_{L^2(\R^d_+)}  + hb \left\| \phi(\cdot,0) \right\|^2_{L^2(\R^{d-1})} \psi_{b/\xi_d}^2 \lk 0 \rk  - \left\| \phi f \right\|^2_{L^2(\R^d_+)}  \, ,
$$
and similar for $ \tr ( \phi \gamma_2 \phi H^+) $.
In the first summand we integrate by parts and use  \eqref{eq:bchs} and \eqref{eq:genef1} to get
$$
\left\| \nabla (\phi f) \right\|^2_{L^2(\R^d_+)} =  \int_{\R^d_+} \lk \frac{|\xi|^2}{h^2} \phi^2 +|\nabla \phi|^2 \rk  \psi^2_{b/\xi_d} \lk x_d \xi_d / h \rk dx  - \frac bh \left\| \phi(\cdot,0) \right\|^2_{L^2(\R^{d-1})} \psi_{b/\xi_d}^2 \lk 0 \rk \, .
$$
We insert this into \eqref{eq:halfbasic} and due to (\ref{eq:int:gradphi}) and Lemma \ref{lem:unifbound} we can estimate 
\begin{equation}
\label{eq:gamma1}
\tr (\phi \gamma_1 \phi H^+) \leq - \frac{4}{(2\pi h )^d}  \int_{\R^d_+} \int_{\R^d_+} \phi^2(x) \lk |\xi|^2 - 1 \rk_-   \psi^2_{b/\xi_d} \lk x_d \xi_d /h \rk dx d\xi + Ch^{-d+2} \, .
\end{equation}

Note that the second summand in (\ref{eq:tracesum}) is zero for $b\geq 0$. For $b< 0$ we use (\ref{eq:bchs}) and (\ref{eq:genef2}) to show that
\begin{align*}
\tr (\phi \gamma_2 \phi H^+) = & \frac{1}{(2\pi h)^{(d-1)}} \int_{ \{|\xi'|^2<b^2+1 \} }   q_b^+[\phi F ] \,  d\xi' \\
= & \frac{1}{(2\pi h)^{(d-1)}}  \int_{\R^{d-1}} \int_{\R^d_+} \phi^2(x) \lk 1 + b^2 -|\xi'|^2 \rk_+ \Psi^2_{b/h}(x_d) dx \, d\xi' \\
& +  \frac{h^2}{(2\pi h)^{(d-1)}} \int_{\{ |\xi'|^2 < b^2+1 \}} \int_{\R^d_+} |\nabla \phi(x)|^2 \Psi_{b/h}^2(x_d) dx \, d\xi' \, .
\end{align*} 
To estimate the last summand we use $\|\Psi_{b/h}\|_\infty^2 \leq -2bh^{-1}$, $\| \Psi_{b/h} \|_2^2 = 1$, and (\ref{eq:int:gradphi}) to obtain 
$$
\int_{\R^d_+} |\nabla \phi(x)|^2 \Psi^2_{b/h}(x_d) dx \leq C \min \{ -b/h,1 \} \, .
$$
Performing the $\xi'$-integration as before yields 
\begin{align}
\nonumber
\tr (\phi \gamma_2 \phi H^+ ) \leq & - \pi C_d \, h^{-d+1} \, (b^2+1)^{(d+1)/2} \int_{\R^d_+} \phi^2(x)  \Psi^2_{b/h} ( x_d ) dx \\
\label{eq:gamma2}
& + Ch^{-d+2} \lk 1 +b_-^{d-1} \min\{h,b_-\} \rk \, .
\end{align}
Here we also used the fact that $1 + (1+b^2)^{(d-1)/2} \min \{ b_-, h \} \leq C (1 + b_-^{d-1}  \min \{ b_-, h \} )$.  
Hence, the upper bound follows from (\ref{eq:tracesum}), (\ref{eq:gamma1}), and (\ref{eq:gamma2}).
\end{proof}

\begin{lemma}
\label{lem:half:mainterm}
Under the conditions of Proposition \ref{pro:half} we have
\begin{align}
\nonumber
& 2C_d \int_{\R^d_+}   \int_0^1 \phi^2(x)  (   1-\xi_d^2 )^{(d+1)/2}   \psi_{b/\xi_d}^2 \lk x_d \xi_d /h \rk d\xi_d dx \\
\label{eq:bdterm1}
& =  \lo  \int_{\R^d_+} \phi^2(x) dx  + C_d \int_0^\infty I_b(t) dt \int_{\R^{d-1}} \phi^2(x',0) dx' h + r_1(h,b) 
\end{align}
with $|r_1(h,b)| \leq C(1+1/|b|)l^{d-2}h^2$ for $b\neq 0$ and $|r_1(h,0)| \leq Cl^{d-2}h^2$. For $b< 0$ we also have
\begin{equation}
\label{eq:bdterm2}
\int_{\R^d_+} \phi^2(x) \Psi^2_{b/h} \lk x_d \rk dx = \int_{\R^{d-1}} \phi^2(x',0) dx' + r_2(h,b)
\end{equation}
with $|r_2(h,b)| \leq Cl^{d-2}hb_-^{-1}$.
Here the constants $C > 0$ depend only on $d$ and $C_\phi$.
\end{lemma}

\begin{proof}
Recall that
$$
\lo = \frac{1}{(2\pi)^d} \int_{\R^d} (|\xi|^2-1)_- d\xi = C_d \int_0^1 (1-\xi_d)^{(d+1)/2} d\xi_d \, .
$$
Hence, 
\begin{align*}
& 2C_d \int_{\R^d_+}   \int_0^1 \phi^2(x)  (   1-\xi_d^2 )^{(d+1)/2}   \psi_{b/\xi_d}^2 \lk x_d \xi_d / h \rk d\xi_d dx   \\
& = \lo \int_{\R^d_+} \phi^2(x) dx + C_d \int_{\R^d_+} \int_0^1 \phi^2(x) ( 1- \xi_d)^{(d+1)/2}  \lk 2 \psi^2_{b/\xi_d} \lk x_d \xi_d / h \rk - 1 \rk d\xi_d dx  \, .
\end{align*}
We insert (\ref{eq:eigenf}) and perform the $\xi'$ integration and see that the right-hand side equals
$$
\lo  \int_{\R^d_+} \phi^2(x) dx  + C_d \int_{\R^d_+} \phi^2(x) I_b \lk \frac{x_d}h \rk dx \, ,
$$
with $I_b$ introduced in Lemma \ref{lem:unifbound}.
To analyze the second term we insert 
\begin{equation}
\label{eq:phiexpand}
\phi^2(x) = \phi^2(x',x_d) = \phi^2(x',0) + \int_0^{x_d} \partial_{s} \phi^2(x',s) ds
\end{equation}
and substitute $x_d = th$. We obtain
\begin{align*}
& 2C_d \int_{\R^d_+}   \int_0^1 \phi^2(x)  (   1-\xi_d^2 )^{(d+1)/2}   \psi_{b/\xi_d}^2 \lk x_d  \xi_d / h \rk d\xi_d dx \\
& =  \lo  \int_{\R^d_+} \phi^2(x) dx + C_d \int_0^\infty I_b(t) dt \int_{\R^{d-1}}  \phi^2(x',0) dx' h \\
& \qquad + C \int_{\R^{d-1}} \int_0^\infty  \int_0^{t h} \partial_s \phi^2(x',s) ds \, I_b(t)  dt dx' h \, .
\end{align*}
Using \eqref{eq:int:gradphi} and the remark at the end of Subsection \ref{sec:half:main} we bound
$$
\left| \int_{\R^{d-1}} \int_0^{t h} \partial_s \phi^2(x',s) ds \, dx' \right| \leq C l^{d-2} h t \,.
$$
The first assertion of the lemma now follows from \eqref{eq:iintbound}.

The second assertion follows similarly by inserting (\ref{eq:phiexpand}) and by definition of $\Psi_b$. 
\end{proof}

Note that the error terms in Lemma \ref{lem:half:mainterm} diverge as $b \to 0$. Hence, we also need the following estimates that yield better results for $|b| \leq C h/l$.

\begin{lemma}
\label{lem:half:smallb}
Under the conditions of Proposition \ref{pro:half} we have
\begin{align*}
& 2C_d \int_{\R^d_+}   \int_0^1 \phi^2(x)  (   1-\xi_d^2 )^{(d+1)/2}   \psi_{b/\xi_d}^2 \lk x_d \xi_d /h \rk d\xi_d dx \\
& =  \lo  \int_{\R^d_+} \phi^2(x) dx  + \frac 14 L^{(1)}_{d-1} \int_{\R^{d-1}} \phi^2(x',0) dx' h + \tilde r_1(h,b) 
\end{align*}
with $|\tilde r_1(h,b)| \leq Cl^{d-2}h^2 \left(1+ l^2 h^{-2} |b|(1+| \ln|b||) \right)$. For $b< 0$ we also have
\begin{equation*}
0 \leq \int_{\R^d_+} \phi^2(x) \Psi^2_{b/h} \lk x_d \rk dx \leq C l^d h^{-1} \min\{b_-,h l^{-1} \} \, .
\end{equation*}
Here the constants $C > 0$ depend only on $d$ and $C_\phi$.
\end{lemma}

\begin{proof}
This proof is a variation of the previous one. Again, we write 
\begin{align}
\nonumber
2C_d \int_{\R^d_+}   \int_0^1 \phi^2(x)  (   1-\xi_d^2 )^{(d+1)/2}   \psi_{b/\xi_d}^2 \lk x_d \xi_d /h \rk d\xi_d dx = & \lo  \int_{\R^d_+} \phi^2(x) dx \\
\label{eq:twoterm}
&  + C_d \int_{\R^d_+} \phi^2(x) I_b \lk \frac{x_d}h \rk dx \, .
\end{align}
We add and subtract $I_0$ to and from $I_b$. According to the previous lemma and Lemma \ref{lem:lz} we have
$$
\left| C_d \int_{\R^d_+} \phi^2(x) I_0 \lk \frac{x_d}h \rk dx - \frac 14 L^{(1)}_{d-1} \int_{\R^{d-1}} \phi^2(x',0) dx' h \right| \leq C l^{d-2} h^2 \,.
$$
Thus, it remains to control
\begin{align*}
 C_d  \int_{\R^d_+} \phi^2(x) \lk I_b \lk \frac{x_d}{h} \rk - I_0 \lk \frac{x_d}{h} \rk \rk dx \,.
\end{align*}
Recalling the definitions of $I_b$ and $I_0$ we see that the absolute value of this term is bounded by
$$
 C  \int_{\R^d_+} \phi^2(x) dx \int_0^1 (1-p^2)^{(d+1)/2} \frac{b^2 + |b|p}{p^2+b^2} dp \leq C l^d |b| (1+ |\ln|b||) \, .
$$
This finishes the proof of the first assertion of the lemma. The second assertion follows similarly as at the end of the proof of Lemma \ref{lem:half:basic}.
\end{proof}

\begin{proof}[Proof of Proposition \ref{pro:half}]
Combining Lemma \ref{lem:half:basic} with \eqref{eq:bdterm1}, \eqref{eq:bdterm2}, and \eqref{eq:lzlem} we obtain the first claim of Proposition \ref{pro:half} with a remainder 
\begin{align*}
|R_{hs}(h,l,b)| & \leq Cl^{d-2}h^{-d+2} \lk 1 + |b|^{-1} + (b^2+1)^{(d+1)/2} b_-^{-1} + b_-^{d-1} \min \{ b_-,h l^{-1} \} \rk \\
& \leq C l^{d-2}h^{-d+2} |b|^{-1} (1+|b|+ b^{d+1}_- ) \, .
\end{align*}

To obtain the second claim we combine Lemma \ref{lem:half:basic} with Lemma \ref{lem:half:smallb}. In this case the remainder is bounded by a constant times 
$$
l^{d-2}h^{-d+2} \lk 1 +l^2h^{-2} |b| (1+ |\ln|b||) + \lk (b^2+1)^{(d+1)/2} l^2 h^{-2}  + b_-^{d-1} \rk \min \{ b_-,h/l \} \rk \, .
$$
For  $|b| \leq h/l \leq 1$ this simplifies to
$$
|R'_{hs}(h,l,b)| \leq C l^{d-2}h^{-d+2} \lk 1 + l^2 h^{-2} |b| ( 1+ |\ln|b|) \rk \, .
$$
This finishes the proof of the proposition.
\end{proof}

\begin{proof}[Proof of Lemma \ref{lem:half}]
Combining Lemma \ref{lem:half:basic} with \eqref{eq:twoterm} we obtain the claim with a remainder bounded by
\begin{align*}
\left|R''_{hs}(h,l,b)\right| = & C_d \int_{\R^d_+} \phi^2(x) \left|I_b \lk \frac{x_d}{h} \rk\right| dx h^{-d} + \pi C_d (b^2+1)^{(d+1)/2}\! \int_{\R^d_+} \phi^2(x) \Psi^2_{b/h}(x_d) dx h^{-d+1} \\
& + C l^{d-2} h^{-d+2} \left( 1+ b_-^{d-1}\min\{b_-, hl^{-1}\}\right) \, .
\end{align*}
In the first term on the right side we substitute $x_d = th$ and use the first inequality in \eqref{eq:iintbound} to bound
$$
\int_0^\infty \int_{\R^{d-1}} \phi^2(x',th) dx' \left| I_b(t)\right| dt  \leq Cl^{d-1} \, .
$$
By Lemma \ref{lem:half:smallb} we also have
$$
0 \leq  \int_{\R^d_+} \phi^2(x) \Psi^2_{b/h}(x_d) dx \leq Cl^d h^{-1} \min\{b_-,h l^{-1} \}
$$
and the proof is complete.
\end{proof}


\section{Local asymptotics close to the boundary}
\label{sec:straight}

Here we show how Proposition \ref{pro:boundary} and Lemma \ref{lem:boundary} follow from the results in Section \ref{sec:half}. We straighten the boundary locally and estimate the operator $H(b)$ given on $\Omega$ in terms of $H^+(b)$ given on the half-space $\R^d_+$.

In this section we work under the conditions of Proposition \ref{pro:boundary}: Let $\phi \in C_0^1(\R^d)$ be supported in a ball of radius $l > 0$ and let inequalities (\ref{eq:hoelder}) and (\ref{eq:int:gradphi}) be satisfied. Then let $B$ denote the open ball of radius $l > 0$, containing the support of $\phi$.
Choose $x_0 \in B \cap \partial \Omega$ and let $\nu_{x_0}$ be the inner normal unit vector at $x_0$. We choose a Cartesian coordinate system such that $x_0 = 0$ and $\nu_{x_0} = (0, \dots, 0, 1)$. 

We now introduce new local coordinates near the boundary. Let $D$ denote the projection of $B$ on the hyperplane given by $x_d =0$.
Since the boundary of $\Omega$ is compact and in $C^1$, there is a constant $C_\Omega>0$, independent of $x_0 \in \partial \Omega$, such that for $0<l\leq C_\Omega^{-1}$ we can find a real function $f \in C^1$, given on $D\subset\R^{d-1}$, satisfying
$$
\partial \Omega \cap B \, = \, \left\{ (x',x_d) \, : \, x' \in D , x_d = f(x') \right\} \cap B \, .
$$
The fact that $\partial \Omega \in C^1$ means that the functions $\nabla f$ corresponding to different points $x_0$ and different values of $l$ share a common modulus of continuity which we denote by $\omega$, that is,
$$
|\nabla f(x') -\nabla f(y')|\leq \omega(|x'-y'|)
$$
for all $x',y'\in D$. We assume that $\omega$ is non-decreasing and we emphasize that $\omega(\delta) \da 0$ as $\delta \da 0$.

The choice of coordinates implies $f(0) = 0$ and $ \nabla f (0) = 0$.
Hence, we can estimate
\begin{equation}
\label{eq:red:fest}
\sup_{x'\in D} |\nabla f(x')| \, \leq \, \sup_{x'\in D} \omega ( |x'|) \, \leq \, \omega (l) \, .
\end{equation}

We introduce new local coordinates given via a diffeomorphism $\varphi \, : \, D \times \R \to \R^d$. We set $y_j \, = \, \varphi_j(x) \, = \, x_j$ for $j = 1, \dots, d-1$ and $y_d \, = \, \varphi_d(x) \, = \, x_d - f(x')$.
Note that the determinant of the Jacobian matrix of $\varphi$ equals $1$ and that the inverse of $\varphi$ is defined on $\textnormal{ran} \, \varphi = D \times \R$. In particular, we get
\begin{equation}
\label{eq:straighten}
\varphi \lk \partial \Omega \cap B \rk \, \subset \, \partial \R^d_+ \, = \, \{
y \in \R^d \, : \, y_d = 0 \} \, .
\end{equation}

Fix $v \in H^1(\Omega)$ with $v \equiv 0$ on $\R^d \setminus \overline B$. For $y \in  \textnormal{ran} \, \varphi$ put $\tilde v (y) = v \circ \varphi^{-1}(y)$ and extend $\tilde v$ by zero to $\R^d$. An explicit calculation shows that the effect of this change of coordinates on the gradient is small:

\begin{lemma}
\label{lem:coordinatechange}
For $v$ and $\tilde v$ defined as above we have $\tilde v \in H^1(\R^d_+)$ and 
$$
\left| \int_{\Omega} |\nabla v(x)|^2 dx - \int_{\R^d_+} |\nabla \tilde v(y)|^2 dy \right| \, \leq  \, C \omega(l)  \int_{\R^d_+} |\nabla \tilde v(y)|^2 dy \, .
$$
\end{lemma}

Based on this estimate we now prove a result from which Proposition \ref{pro:boundary} follows.
For $\phi \in C_0^\infty(\R^d)$ supported in $B$ define $\tilde \phi = \phi \circ \varphi^{-1}$ on $\textnormal{ran} \, \varphi = D \times \R$ and extend it by zero to $\R^d$. It follows that $\tilde \phi \in C_0^1(\R^d)$ and $\| \nabla \tilde \phi \|_\infty \leq Cl^{-1}$ hold, with $C$ depending only on $C_\phi$ and $\omega$.
We set $b^- = \inf_{x \in \partial \Omega \cap B} b(x)$ and $b^+ =  \sup_{x \in \partial \Omega \cap B} b(x)$ and note that $(b^+)_- \leq (b^-)_- \leq b^s$, where $b^s$ was introduced in \eqref{eq:bpm}. We also recall the notation $H^+(b^\pm)$ introduced in Section~\ref{sec:half}.

\begin{lemma}
\label{lem:str}
Under the conditions of Proposition \ref{pro:boundary} there is a constant $C_\Omega >0$ depending only on $\Omega$ such that for  $0 < l \leq C_\Omega^{-1}$ and $0<h\leq l$ we have
\begin{align}
\label{eq:str:lem1}
\nonumber
& \tr ( \tilde \phi H^+(b^+) \tilde \phi )_- -  C l^d h^{-d} \omega(l) \lk 1 +(b^+)_-^{d+1} h l^{-1} \rk \\
\nonumber
&\leq \tr ( \phi H(b) \phi )_- \\
&\leq \tr ( \tilde \phi H^+(b^-) \tilde \phi )_-  + C l^d h^{-d} \omega(l) \lk 1 +(b^-)_-^{d+1} h l^{-1} \rk \, .
\end{align}
Moreover,  
\begin{equation}
\label{eq:str:lem2}
\int_\Omega \phi^2(x) \, dx  \, = \, \int_{\R^d_+} \tilde \phi^2(y) \, dy \, ,
\end{equation}
\begin{equation}
\label{eq:str:lem4}
\left| \int_{\partial \Omega}\phi^2 (x) d\sigma(x)  - \int_{\R^{d-1}} \tilde \phi^2 (y',0) dy' \right| \leq Cl^{d-1} \omega(l)^2 \, ,
\end{equation}
and
\begin{align}
\nonumber
&\left| \int_{\partial \Omega} L^{(2)}_d (b(x)) \phi^2 (x) d\sigma(x)  -  L^{(2)}_d(b^\pm) \int_{\R^{d-1}} \tilde \phi^2 (y',0) dy' \right| \\
\label{eq:str:lem3}
& \leq \,   Cl^{d-1}  \lk (1+(b^\pm)_-^{d+1})   \omega(l)^2 + (1+(b^\pm)_-^d )   \beta(l) \rk  \, .
\end{align}
\end{lemma}

\begin{proof}
The definition of $\tilde \phi$ and the fact that $\textnormal{det} J\varphi = 1$ immediately give (\ref{eq:str:lem2}). In view of (\ref{eq:red:fest}) we can estimate
$$
\int_{\partial\Omega} \phi^2(x) d\sigma(x) =  \int_{\R^{d-1}}   \tilde \phi^2(y',0) \sqrt{1+|\nabla f|^2 }  dy' \leq  \int_{\R^{d-1}} \tilde \phi^2(y',0)  dy' +C l^{d-1} \omega(l)^2 \, .
$$
This proves \eqref{eq:str:lem4}. Using the fact that $|L^{(2)}_d(b^\pm)| \leq C (1+(b^\pm)_-^{d+1})$ we find
\begin{align*}
& \left| \int_{\partial \Omega} L^{(2)}_d (b(x)) \phi^2 (x) d\sigma(x) -  L^{(2)}_d(b^\pm) \int_{\R^{d-1}} \tilde \phi^2 (y',0) dy' \right| \\
& \leq \, \int_{\partial \Omega} \left|  L^{(2)}_d (b(x)) -  L^{(2)}_d (b^\pm) \right| \phi^2(x) d\sigma(x) + Cl^{d-1} \omega(l)^2 \lk 1+(b^\pm)_-^{d+1} \rk \, . 
\end{align*}
The continuity of $b$, see (\ref{eq:hoelder}), and the fact that $|\frac{d}{db} L^{(2)}_d(b)| \leq C (1+b_-^d)$ imply
$$
\left| L^{(2)}_d(b^\pm) -  L^{(2)}_d(b(x)) \right| \leq C \beta(l) \lk 1 + (b^\pm)_-^d \rk \, .
$$
Inserting this into the estimate above gives (\ref{eq:str:lem3}). 

To prove (\ref{eq:str:lem1}) we first note that the variational principle implies
$$
\tr \lk \phi H(b^+) \phi \rk_-  \leq \tr \lk \phi H(b) \phi \rk_- \leq \tr \lk \phi H(b^-) \phi \rk_- \, .
$$
Thus it remains to show that
\begin{equation}
\label{eq:strproof}
\left| \tr \lk \phi H(b^\pm) \phi \rk_- - \tr ( \tilde \phi H^+(b^\pm) \tilde \phi )_- \right| \leq Cl^d h^{-d} \omega(l) \lk 1 + (b^\pm)_-^{d+1} h l^{-1} \rk \, .
\end{equation}
To this end choose $v$ and $\tilde v$ as in Lemma \ref{lem:coordinatechange}. First we estimate 
\begin{equation}
\label{eq:formest1}
\int_{\partial \Omega}  |v(x)|^2 d\sigma(x) =  \int_{\R^{d-1}}  |\tilde v(y',0)|^2 \sqrt{1+|\nabla f|^2 }  dy' \geq \int_{\R^{d-1}} |\tilde v (y',0)|^2 dy' \, .
\end{equation}
and using (\ref{eq:red:fest})
\begin{equation}
\label{eq:formest2}
\int_{\partial \Omega}  |v(x)|^2 d\sigma(x) \leq (1+C \omega(l)^2 ) \int_{\R^{d-1}} |\tilde v (y',0)|^2 dy' \, .
\end{equation}
By decreasing, if necessary, the constant $C_\Omega$ from the beginning of this section we may now assume that $l > 0$ is small enough such that $2C \omega(l) \leq 1/2$ holds. Then Lemma \ref{lem:coordinatechange}, \eqref{eq:str:lem2}, and \eqref{eq:formest1} imply, for $b^\pm \geq 0$, 
\begin{align}
\nonumber
q_{b^\pm}[v] \geq  & \, (1-C\omega(l)) h^2 \int_{\R^d_+} |\nabla \tilde v(y)|^2 dy + h b^\pm \int_{\R^{d-1}} |\tilde v (y',0)|^2 dy' - \int_{\R^d_+} |\tilde v(y)|^2 dy \\
\nonumber
= & \, (1-2C\omega(l)) q^+_{b^\pm}[\tilde v] \\
\nonumber
&  \, + 2C\omega(l) \lk  \frac{h^2}2 \int_{\R^d_+} |\nabla \tilde v(y)|^2 dy +  h b^\pm \int_{\R^{d-1}} |\tilde v (y',0)|^2 dy' -  \int_{\R^d_+} |\tilde v(y)|^2 dy \rk \\ 
\label{eq:formdif}
= & \, (1-2C\omega(l)) q^+_{b^\pm}[\tilde v] +2C\omega(l)  \tilde q^+_0[\tilde v]  \, ,
\end{align}
where $\tilde q^+$ is the same form as $q^+$ but with $h$ replaced by $h/\sqrt 2$.
For $b^\pm < 0$ we get, using \eqref{eq:formest2},
\begin{align}
\nonumber
q_{b^\pm}[v] \geq \, & (1-C\omega(l)) h^2 \int_{\R^d_+} |\nabla \tilde v(y)|^2 dy \\
\nonumber
&+ (1+C\omega(l)^2) h b^\pm \int_{\R^{d-1}} |\tilde v (y',0)|^2 dy' - \int_{\R^d_+} |\tilde v(y)|^2 dy \\
\label{eq:formdif2}
\geq & (1-2C\omega(l)) q^+_{b^\pm}[\tilde v] +2C\omega(l)  \tilde q^+_{Cb^\pm}[\tilde v]  \, .
\end{align}

To deduce estimates for $\tr \lk \phi H(b^\pm) \phi \rk_- $ we recall the variational principle
$$
- \tr \lk \phi H(b^\pm) \phi \rk_- \, = \, \inf_{0 \leq \gamma \leq 1} \tr \lk
\phi \gamma \phi H(b^\pm) \rk \, ,
$$
where we can assume that the infimum is taken over trial density matrices $\gamma$ supported in $\overline{B} \times
\overline{B}$. Fix such a $\gamma$. 
For $y$ and $z$ from  $D \times \R$ set
$$
\tilde \gamma (y,z) \, = \, \gamma \lk \varphi^{-1}(y), \varphi^{-1}(z) \rk \, ,
$$
so that $0 \leq \tilde \gamma \leq  1$ holds. Moreover, the range of $\tilde \gamma$ belongs
to the form domain of $\tilde \phi H^+(b^\pm) \tilde \phi$.

First, we assume $b^\pm < 0$. According to (\ref{eq:formdif2}) it follows that 
\begin{align*}
\tr \lk \phi \gamma \phi H(b^\pm) \rk  \geq \, & \tr \lk \tilde \phi \tilde \gamma
\tilde \phi \lk  (1-2C\omega(l))H^+(b^\pm) + 2C \omega(l) \tilde H^+(C b^\pm ) \rk \rk \\
\geq & - (1-2C \omega(l)) \tr \lk \tilde \phi H^+(b^\pm) \tilde \phi \rk_- - 2C \omega(l) \tr \lk \tilde \phi \tilde H^+(C b^\pm) \tilde \phi \rk_- \, ,
\end{align*}
where the operator $\tilde H^+$ is generated by the form $\tilde q^+$. This implies
$$
\tr (\phi H(b^\pm)\phi )_-  \leq  \tr ( \tilde \phi  H^+(b^\pm) \tilde \phi )_- + 2C \omega(l) \tr \lk \tilde \phi \tilde H^+(C b^\pm) \tilde \phi \rk_- 
$$
and Corollary \ref{cor:half} yields
$$
\tr ( \phi H(b^\pm) \phi )_- \, \leq \, \tr (\tilde \phi H^+(b^\pm) \tilde \phi )_- + C  l^d  h^{-d} \omega(l)  \lk 1 +(b^\pm)^{d+1} h/l \rk
$$
for $b^\pm < 0$.

In the same way we can treat non-negative $b^\pm$ using (\ref{eq:formdif}) and we obtain the lower bound in  (\ref{eq:strproof}). Finally, by interchanging the roles of $H(b^\pm)$ and $H^+(b^\pm)$, we get an analogous upper bound and the proof of Lemma \ref{lem:str} is complete.
\end{proof}

\begin{proof}[Proof of Proposition \ref{pro:boundary} and Lemma \ref{lem:boundary}]
The assertions follow from Lemma~\ref{lem:str} together with Proposition~\ref{pro:half}.
\end{proof}

If we combine the estimates of Proposition \ref{pro:bulk}, Corollary \ref{cor:half}, and Lemma \ref{lem:str} we obtain the following simple bound that is useful to estimate error terms.

\begin{corollary}
\label{cor:simple}
There is a constant $C_\Omega > 0$ with the following property. Let $\phi \in C_0^\infty$ be supported in a ball of radius $l > 0$ and let \eqref{eq:int:gradphi} be satisfied. Assume that $b$ is a real constant independent of $x$.

Then for $0 < l \leq C_\Omega^{-1}$ and $0<h\leq l$ the estimate
$$
\tr \lk \phi H(b) \phi \rk_- \leq C l^d h^{-d} \lk 1 + b_-^{d+1}h l^{-1} \rk
$$
holds with a constant $C > 0$ depending only on $d$, $C_\phi$ and $\omega$.
\end{corollary}


\section{Localization}
\label{sec:loc:s1} 

In this section we construct the family of localization functions $(\phi_u)_{u \in \R^d}$ and prove Proposition~\ref{pro:loc:s1}. The key idea is to choose the localization depending on the distance to the complement of $\Omega$, see \cite[Theorem 17.1.3]{Hoe85} and \cite{SolSpi03} for a continuous version of this method.

Fix a real-valued function $\phi \in C_0^\infty(\R^d)$ with support in $\{|x| < 1\}$ and $\| \phi \|_2 = 1$. For $u, x \in \R^d$ let $J(x,u)$ be the Jacobian of the map $u \mapsto (x-u)/l(u)$. We define 
$$
\phi_u(x) \, = \, \phi \lk \frac{x-u}{l(u)} \rk \sqrt{J(x,u)} \, l(u)^{d/2} \, ,
$$
such that $\phi_u$ is supported in $\{ x \, : \, |x-u| < l(u) \}$.
By definition, the function $l(u)$ is smooth and satisfies $0 < l(u) \leq 1/2$ and $\left\| \nabla l \right\|_\infty \leq 1/2$. Therefore, according to \cite{SolSpi03}, the functions $\phi_u$ satisfy (\ref{eq:int:grad:s1}) and (\ref{eq:int:unity:s1}) for all $u \in \R^d$. 

To prove the upper bound in Proposition \ref{pro:loc:s1}, put
$$
\gamma \, = \, \int_{\R^d} \phi_u \, \lk \phi_u H(b) \phi_u \rk_-^0 \, \phi_u \, l(u)^{-d} \, du \, .
$$
Obviously, $\gamma \geq 0$ holds and in view of (\ref{eq:int:unity:s1}) also $\gamma \leq 1$, hence, by a variant of the variational principle discussed in the appendix,
$$
- \tr (H(b))_- \, \leq \, \tr \lk \gamma H(b) \rk \, =  \, -  \int_{\R^d} \tr \lk \phi_u H(b) \phi_u \rk_- l(u)^{-d} \, du \, .
$$

To prove the lower bound we use the IMS-formula. For $\phi \in C_0^\infty(\R^d)$ and $v \in H^1(\Omega)$ we have
$$
\frac 12 \nabla v \cdot \nabla \lk \phi^2 \overline v \rk  + \frac 12 \nabla \overline{v} \cdot \nabla \lk \phi^2 v \rk \, = \, \left| \nabla \lk \phi v \rk \right|^2  -  \left| \nabla \phi \right|^2 |v|^2  \, . 
$$
Combining this identity with the partition of unity (\ref{eq:int:unity:s1}) yields
\begin{equation}
\label{eq:loc:ims}
q_b[v] = \int_{\R^d} \lk q_b \left[ \phi_u v \right] - \lk v ,  h^2 (\nabla \phi_u)^2 v \rk_{L^2(\Omega)} \rk l(u)^{-d} \, du \, .
\end{equation}
Using (\ref{eq:int:grad:s1}) and (\ref{eq:int:unity:s1}) one can show \cite{SolSpi03}, for every $x \in \R^d$, 
$$
\int_{\R^d} (\nabla \phi_u)^2(x) l(u)^{-d} \, du \, \leq \, C \int_{\R^d} \phi_{u}^2(x) \, l(u)^{-d-2} \, du \, .
$$
We insert this into (\ref{eq:loc:ims}) and deduce
\begin{equation}
\label{eq:loc}
\tr \lk H(b) \rk_- \, \leq \, \int_{\Omega^*} \tr \lk \phi_u \lk H(b) - C h^2 l(u)^{-2} \rk \phi_u \rk_- \, l(u)^{-d} \, du \, ,
\end{equation}
where $\Omega^* = \{u \in \R^d \, : \, \textnormal{supp} \phi_u \cap \Omega \neq \emptyset \}$.
For any $u \in \R$, let $\rho_u$ be another parameter $0< \rho_u < 1$ and estimate
$$
\textnormal{Tr} \lk \phi_u ( H(b) - C h^2  l(u)^{-2} ) \phi_u \rk_-  \leq  \textnormal{Tr} \lk \phi_u H(b) \phi_u \rk_- +  \textnormal{Tr} \lk \phi_u (\rho_u H(b)-C h^2 l(u)^{-2}  ) \phi_u \rk_- \, .
$$

We now claim that choosing $\rho_u$ proportional to $h^2 l(u)^{-2}$ yields
\begin{equation}
\label{eq:loc:error}
\textnormal{Tr} \lk \phi_u (H(b)- C h^2  l(u)^{-2}) \phi_u \rk_- \leq \textnormal{Tr} \lk \phi_u  H(b) \phi_u \rk_- + C \frac{l(u)^{d-2}}{ h^{d-2}} \lk 1 + \frac{(b_m)_-^{d+1} h}{ l(u)} \rk \, .
\end{equation}
To see this, let us write $\tau_u = \rho_u / (\rho_u+Ch^2 l(u)^{-2})$ and note that $\tau_u < 1$ and 
$$
\textnormal{Tr} \lk \phi_u (\rho_u H(b)-C h^2 l(u)^{-2}  ) \phi_u \rk_- = Ch^2 l(u)^{-2} (1-\tau_u)^{-1} \tr ( \phi_u \tilde H( \sqrt{\tau_u} b) \phi_u )_- \, .
$$
Here $\tilde H$ is generated by the same quadratic form as $H$ but with $h$ replaced by $\sqrt{\tau_u} h$. 
If $\phi_u \cap \partial \Omega \neq \emptyset$, we have $l_0 /4 \leq l(u) \leq l_0/\sqrt 3$, see (\ref{eq:lulower}) and  (\ref{eq:int:luup:s1}), and we can apply Corollary~\ref{cor:simple} to estimate 
$$
\tr ( \phi_u \tilde H( \sqrt{\tau_u} b) \phi_u )_- \leq Cl(u)^d h^{-d} \tau_u^{-d/2} \lk 1+ (b_m)_-^{d+1} h l(u)^{-1} \rk \, .
$$
With our choice of $\rho_u$ proportional to $h^2 l(u)^{-2}$ we find that $\tau_u$ is order one and (\ref{eq:loc:error}) follows.
If $\phi_u \in C_0^\infty$ we can argue similarly by using the lower bound in Proposition \ref{pro:bulk} and get
\begin{equation}
\label{eq:loc:error2}
\textnormal{Tr} \lk \phi_u (H(b)- C h^2  l(u)^{-2}) \phi_u \rk_- \leq \textnormal{Tr} \lk \phi_u  H(b) \phi_u \rk_- + C \frac{l(u)^{d-2}}{ h^{d-2}}  \, .
\end{equation}

Finally, we insert (\ref{eq:loc:error}) and (\ref{eq:loc:error2})  into (\ref{eq:loc}) and arrive at
\begin{align*}
\tr \lk H(b)\rk_- \, \leq \, & \int_{\Omega^*} \tr \lk \phi_u  H(b) \phi_u \rk_- l(u)^{-d} du + C h^{-d+2}  \int_{\Omega \setminus U} l(u)^{-2} du \\
&+  C h^{-d+2}  \int_{U} \lk l(u)^{-2} + (b_m)_-^{d+1} h l(u)^{-3} \rk du \, ,
\end{align*}
where $U=\{ u \in \R^d \, : \, \partial \Omega \cap B_u \neq \emptyset \} $.  
Thus the claim of Proposition \ref{pro:loc:s1}  follows from  (\ref{eq:int:U1}) and (\ref{eq:int:U2}).

\appendix

\section{A geometric lemma}

In the proofs of Theorem \ref{thm:1} and Theorem \ref{thm:large} we used the following estimate.

\begin{lemma}
\label{lem:parallel}
For every domain $\Omega\subset\R^d$ with $\partial\Omega\in C^1$ there is a constant $C$ with the following property. For every $0<l_0\leq 1$ and $u \in \R^d$ let $l(u)$ be defined as in \eqref{eq:ludef} by 
$$
l(u) = \frac 12 \lk 1 + \lk \textnormal{dist}(u, \R^d \setminus \Omega )^2+l_0^2 \rk^{-1/2} \rk^{-1} \,.
$$
Then for any relatively open $N \subset \partial \Omega$ the set 
$$
U^* = \left\{u \in \R^d \, : \, \textnormal{dist}(u, \partial \Omega) < l(u) \, \wedge \,  \textnormal{dist}(u,\partial \Omega \setminus N ) > l(u) \right\}
$$
satisfies
$$
\limsup_{l_0 \da 0} \frac 1{l_0} |U^*|_d \leq C \sigma(N) \, .
$$
Here $|\cdot|_d$ denotes the $d$-dimensional Lebesgue measure on $\R^d$ and $\sigma(\cdot)$ denotes the $d-1$-dimensional surface measure on $\partial \Omega$.
\end{lemma}

\begin{proof}
We split $U^*$ into two parts $U^*_i = U^* \cap \Omega$ and $U^*_o = U^* \cap \R^d \setminus \Omega$ and we prove the assertion separately for each of them. We begin with $U^*_i$. Note that for $u \in \Omega$ we have $ \textnormal{dist}(u,\R^d \setminus  \Omega) = \textnormal{dist}(u,\partial \Omega)$. We first argue that there is a constant $L_{l_0}$ such that
\begin{equation}
\label{eq:ul}
U^*_i = \left\{u \in \Omega \, : \, \textnormal{dist}(u, \partial \Omega) < L_{l_0} \, \wedge \, \textnormal{dist}(u,\partial \Omega \setminus N ) > l(u) \right\}
\end{equation}
and such that $l_0/4\leq L_{l_0}\leq l_0/\sqrt 3$.

To prove \eqref{eq:ul} let us consider the function
$$
F_{l_0}(x) = \frac 12 \lk 1+ \lk x^2+l_0^2 \rk^{-1/2} \rk^{-1} - x \, , \quad x \geq 0 \, .
$$
This function is continuously differentiable and satisfies $F_{l_0}(0) = l_0/(2(l_0+1)) > 0$, $F_{l_0}(x) \leq 0$ for $x \geq 1/2$, and 
$$
F_{l_0}'(x) = \frac x2 \lk x^2+l_0^2 \rk^{-1/2} \lk 1+ \lk x^2+l_0^2 \rk^{1/2} \rk^{-2} - 1 \leq - \frac 12
$$
for all $x \geq 0$. Hence, there is a unique $L_{l_0} \in (0,1/2]$ with $F_{l_0}(L_{l_0}) = 0$. Moreover, since $F_{l_0}(l_0/4)<0<F_{l_0}(l_0/\sqrt 3)$, we have $l_0/4< L_{l_0}< l_0/\sqrt 3$.

By definition, all $u \in \Omega$ with $ \textnormal{dist}(u,\partial \Omega) = L_{l_0}$ satisfy $F_{l_0}( \textnormal{dist}(u,\partial \Omega) ) = 0$, thus $l(u) =  \textnormal{dist}(u,\partial \Omega) = L_{l_0}$. The fact that $F_{l_0}$ is decreasing shows that the inequality $ \textnormal{dist}(u,\partial \Omega) < L_{l_0}$ implies $F_{l_0}( \textnormal{dist}(u,\partial \Omega)) > 0$, thus $\textnormal{dist}(u,\partial \Omega)<l(u)$. Similarly, the inequality $\textnormal{dist}(u,\partial \Omega) < l(u)$ implies $ \textnormal{dist}(u,\partial \Omega) < L_{l_0}$. This proves \eqref{eq:ul}.

Our next step is to fix an $0<\epsilon <1$ and to decompose $U^*_i = U^*_> \cup U^*_\epsilon$ with
\begin{align*}
U^*_> &= \left\{u \in \Omega \, : \, \textnormal{dist}(u, \partial \Omega) < (1-\epsilon)L_{l_0} \, \wedge \, \textnormal{dist}(u,\partial \Omega \setminus N ) > l(u) \right\} \, \\
U^*_\epsilon &= \left\{u \in \Omega \, : \, (1-\epsilon)L \leq \textnormal{dist}(u, \partial \Omega) < L_{l_0} \, \wedge \, \textnormal{dist}(u,\partial \Omega \setminus N ) > l(u) \right\} \,.
\end{align*}
Thus,
\begin{equation*}
|U_i^*|_d \leq |U_>^*|_d +|U_\epsilon^*|_d \, .
\end{equation*}
The second term on the right side can easily be bounded,
\begin{equation*}
|U_\epsilon^*|_d \leq \left| \left\{ u \in \Omega \, : \, (1-\epsilon)L \leq \textnormal{dist}(u, \partial \Omega) < L \right\} \right|_d \leq \int_{(1-\epsilon)L_{l_0}}^{L_{l_0}} \sigma \lk \partial \Omega_t \rk dt \leq C l_0 \epsilon \, .
\end{equation*}
Here we wrote $\partial \Omega_t = \{ u \in \Omega : \textnormal{dist}(u, \partial \Omega)  = t \}$ and used the facts that $\sigma(\partial \Omega_t)$ is uniformly bounded and that $L_{l_0} \leq l_0/\sqrt 3$.

After these steps we have reduced the lemma to proving that
\begin{equation}
\label{eq:ugoal}
\limsup_{l_0 \da 0} \frac 1{l_0} |U^*_>|_d \leq C \sigma(N)
\end{equation}
with a constant $C$ independent of $\epsilon$. To do so we start from the representation
\begin{equation}
\label{eq:ufub}
|U^*_>|_d = \int_0^{(1-\epsilon)L_{l_0}} \sigma(U_t^*) \, dt \, ,
\end{equation}
where
$$
U^*_t = \left\{u \in \Omega \, : \, \textnormal{dist}(u, \partial \Omega) = t \, \wedge \, \textnormal{dist}(u,\partial \Omega \setminus N ) > l(u) \right\} \, , \quad 0 \leq t < (1-\epsilon) L_{l_0} \, .
$$

Recall that every $u\in U^*$ and, in particular, every $u\in U^*_>$ satisfies $\dist(u, \partial \Omega)<l(u)$. We now claim that for every $0<\epsilon<1$ and every $0<l_0\leq 1$ there is an $r>0$ such that every $u\in U^*_>$ satisfies
$$
l(u) > \dist(u,\partial \Omega) + r \,.
$$
This follows again from the monotonicity and continuity of the function $F_{l_0}$. Indeed, we can set $r= F_{l_0} ((1-\epsilon) L_{l_0})$.

We consider the set
$$
\tilde N := \bigcup_{0<t<(1-\epsilon)L_{l_0}}\ \bigcup_ {u \in U_t^*} \ \bigcup_{x\in \partial \Omega,\, |x-u|=t} \{y\in\R^d:\ |y-x|<r\} \cap \partial \Omega
$$
and show that
\begin{equation}
\label{eq:nntilde}
\tilde N \subset N
\end{equation}
and
\begin{equation}
\label{eq:measurezero}
\sigma( \partial \tilde N ) = 0 \,.
\end{equation}

To prove \eqref{eq:nntilde} let $0<t<(1-\epsilon)L_{l_0}$, $x,y\in \partial\Omega$ with $|x-y|<r$ and $u\in U^*_t$ with $|x-u|=t$. Then
$$
|y-u| \leq |y-x|+ |x-u| < r + \dist(u,\partial\Omega) < l(u) \,.
$$
Since $\dist(u,\partial\Omega\setminus N)>l(u)$ by the definition of $U^*_t$, we infer that $y\in N$. This proves \eqref{eq:nntilde}.

To prove \eqref{eq:measurezero} we note that $\tilde N$ satisfies the following uniform interior ball condition. For each $y \in \partial \tilde N$ there is an open ball $B\subset\R^d$ of radius $r$ such that $y \in \partial B$ and $B \cap \partial \Omega \subset \tilde N$. In order to prove \eqref{eq:measurezero} we introduce local coordinates similarly as in Section \ref{sec:straight}. In this way we are reduced to the situation where $\tilde N$ is a subset of $\R^{d-1}$ satisfying a uniform interior ball condition (with a possibly smaller radius). The claim \eqref{eq:measurezero} follows from Lemma \ref{lem:lebesgue} below.

The definition of $\tilde N$ easily implies that
$$
U^*_t \subset \tilde U^*_t := \left\{u \in \Omega \, : \, \dist(u, \partial \Omega) = t \, \wedge \, \dist(u, \tilde N ) = t \right\}
$$
for all $0 \leq t < (1-\epsilon)L_{l_0}$. Moreover, we can estimate with a constant depending only on $\Omega$
\begin{align*}
\sigma( \tilde U^*_t ) &\leq C \lk \sigma ( \tilde N ) + \sigma ( \{ x \in \partial \Omega \setminus \tilde N  :  \textnormal{dist}(x,\tilde N) < t \} ) \rk \\
&\leq C \lk \sigma (  N ) + \sigma ( \{ x \in \partial \Omega \setminus \tilde N  :  \textnormal{dist}(x,\tilde N) < l_0 \} ) \rk \,.
\end{align*}
The second bound used \eqref{eq:nntilde} as well as $(1-\epsilon)L_{l_0} \leq (1-\epsilon) l_0/\sqrt 3 \leq l_0$. Thus, from \eqref{eq:ufub},
\begin{equation*}
|U_>^*|_d \leq Cl_0 \lk \sigma (  N ) + \sigma ( \{ x \in \partial \Omega \setminus \tilde N  :  \textnormal{dist}(x,\tilde N) < l_0 \} ) \rk \, .
\end{equation*}
Therefore, in order to prove \eqref{eq:ugoal}, it remains to estimate
$$
\sigma (\{ x \in \partial \Omega \setminus \tilde N  : \textnormal{dist}(x,\tilde N) < l_0 \} ) = \int_{\partial \Omega} \chi_{l_0} (y) d\sigma(y) \, ,
$$
where $\chi_{l_0}$ denotes the characteristic function of $\{x \in \partial \Omega \setminus \tilde N  : \textnormal{dist}(x,\tilde N) < l_0 \}$. We note that $\lim_{l_0 \da 0} \chi_{l_0} = \chi_{\partial \tilde N}$ pointwise. Thus, the dominated convergence theorem and \eqref{eq:measurezero} imply that
$$
\lim_{l_0 \da 0} \sigma (\{ x \in \partial \Omega \setminus \tilde N  : \textnormal{dist}(x,\tilde N) < l_0 \} ) = 0 \, .
$$
This completes the proof of \eqref{eq:ugoal}.

For $U^*_o$ we get an analoguous bound by following the same strategy. In this case the estimates are somewhat simpler since, for $u \in \R^d \setminus \Omega$, we have $l(u) \equiv \frac 12 l_0/(l_0+1)$ and this plays the role of $L_{l_0}$.
\end{proof}

\begin{lemma}
\label{lem:lebesgue}
Let $A \subset \R^n$ be bounded. Assume that there is $\rho > 0$ such that for each $x \in \partial A$ there is a ball $B \subset \R^n$ of radius $\rho$ with $x \in \partial B$ and $B \subset A$. Then $| \partial A|_{n} = 0$.
\end{lemma}

\begin{proof}
Let $\delta>0$ be a constant to be specified later and put $l_m = \delta \rho 5^{-m}$ for $m \geq 0$. We denote by $\mathcal Q_m$ the collection of open cubes of side length $l_m$ centered at points in $(l_m\Z)^n$. Let $\mathcal C_m$ be the collection of those cubes in $\mathcal Q_m$ that intersect both $A$ and $\R^n\setminus A$. Since $A$ is bounded, $\nu_m := \#\mathcal C_m$ is finite. We claim that for all sufficiently small $\delta>0$ there is a constant $M<5^n$ such that for all $m\geq 1$
\begin{equation}
\label{eq:intersectionno}
\nu_m \leq M \nu_{m-1} \,.
\end{equation}
Deferring the proof of this bound for the moment we now explain why it implies the lemma. First, we iterate \eqref{eq:intersectionno} to learn that $\nu_m \leq M^m \nu_0$. Thus, since $\partial A \subset\bigcup_{Q\in\mathcal C_m} \overline Q$ for any $m$, we conclude that
$$
|\partial A|_n \leq \sum_{Q\in\mathcal C_m} |Q|_n = l_m^n \nu_m \leq \delta^n \rho^n (5^{-n } M)^m \nu_0 \to 0
\qquad\text{as}\ m\to\infty \,.
$$
This proves $|\partial A|_n=0$ and we are left with showing \eqref{eq:intersectionno}.

To do so, we fix $m\geq 1$ and an arbitrary cube $Q\subset\mathcal C_{m-1}$. When passing from $m-1$ to $m$, this cube is subdivided into $5^n$ cubes in $\mathcal Q_m$. We shall show that if $\delta>0$ is sufficently small then at least one of these cubes of side length $l_m$ does not belong to $\mathcal C_m$ (i.e., does not intersect both $A$ and $\R^n\setminus A$). This will imply \eqref{eq:intersectionno} with $M=5^n -1$.

Consider the cube $Q'\in \mathcal Q_{m}$ in the center of $Q$. If this cube does not belong to $\mathcal C_m$ we are done. Thus, we may assume that $Q'$ intersects both $A$ and $\R^n\setminus A$. Because of our assumption on $\partial A$ there is an open ball $B$ of radius $\rho$ such that $B\subset A$ and $\partial B\cap Q' \neq \emptyset$. We now make use of the following

\emph{Claim.} There is a constant $C_n>0$ such that if $B\subset\R^n$ is an open ball of radius $r\geq C_n$ with $\overline B\cap Q\neq\emptyset$, where $Q=(-1/2,1/2)^n$, then $\gamma+Q\subset B$ for some $\gamma\in\Z^n$ with $|\gamma|_\infty \leq 2$.

Indeed, one can take $C_n=\max\{\sqrt n, n/2\}$. The proof of this claim uses only elementary geometric facts and is omitted. 

By a rescaled version of the claim we infer that, under the assumption that $\rho\geq C_n l_m$, there is a cube which is contained in $B$ and whose center is at most an $\infty$-distance $2l_m$ away from that of $Q'$. Since $Q'$ lies in the center of $Q$ this cube is also contained in $Q$. Moreover, since it is contained in $B$, it is also contained in $A$ and, therefore, does not belong to $\mathcal C_m$.

Finally, we argue that for all $\delta>0$ small enough the assumption $\rho \geq C_n l_m$ is satisfied for all $m\geq 1$. Indeed, this assumption is equivalent to $1\geq C_n \delta 5^{-m}$, which holds uniformly in $m\geq 1$ provided we choose $\delta\leq 5 C_n^{-1}$. This completes the proof.
\end{proof}


\section{A variant of the variational principle and a sharp bound on $\tr(-\Delta_b - \Lambda)_-$}
\label{ap:var}

Here we mention the following extension of the variational principle that we used in the proof of Proposition \ref{pro:half}.

Let $(M,\mu)$ be a measure space and let $(f_\alpha)_{\alpha \in M}$ be a measurable family of functions in a separable Hilbert space $\mathcal G$, such that 
\begin{equation}
\label{eq:ap:1}
\int_{M} \left| \lk \psi , f_\alpha \rk \right|^2 d\mu(\alpha) \leq \| \psi \|^2
\end{equation}
for all $\psi \in \mathcal G$. Assume that $A$ is a self-adjoint, lower semibounded operator in $\mathcal G$ with quadratic form $a$ such that 
\begin{equation}
\label{eq:ap:2}
f_\alpha \in \mbox{dom}[a]
\end{equation}
for all $\alpha \in M$.

Let the operator $\gamma$ in $\mathcal G$ be given by $\gamma \psi = \int_M ( f_\alpha, \psi) f_\alpha d\mu(\alpha)$. Then $\gamma$ satisfies $0 \leq \gamma \leq 1$. Let us introduce the notation
$$
\tr A \gamma = \int_M a\left[ f_\alpha \right] d\mu(\alpha) \, .
$$
Then we have
\begin{equation}
\label{eq:traces}
- \tr A_- \leq \tr A\gamma \, ,
\end{equation}
provided $\int_M a[f_\alpha]_- d\mu(\alpha) < \infty$.

Let us illustrate these notions by adding the following sharp estimate, a simple form of the upper in Proposition \ref{pro:half}, which is based on a method introduced in \cite{Kro92}.
Here we only assume that the boundary of $\Omega \subset \R^d$ is Lipschitz continuous and that $-\Delta_b$ is generated by the quadratic form given in \eqref{eq:basicform}.
\begin{proposition}
For $\phi \in C_0^1(\R^d)$ and $\Lambda > 0$
\begin{align*}
\tr \lk \phi \lk -\Delta_b - \Lambda \rk \phi \rk_- \geq \, & L^{(1)}_d \Lambda^{1+d/2} \int_\Omega |\phi(x)|^2 dx \\
&- \frac{\omega_d}{(2\pi)^d} \Lambda^{d/2} \lk \int_{\partial \Omega} b(x) |\phi(x)|^2 d\sigma(x) + \int_\Omega |\nabla \phi|^2 dx \rk \, .
\end{align*}
\end{proposition}

\begin{proof}
To adopt the notation introduced above, we set $\mathcal G = L^2(\Omega)$, $M = \{\xi \in \R^d : |\xi|^2 \leq \Lambda \}$ and $\mu$ to be Lebesgue measure. If we choose $f_\xi(x) = (2\pi)^{-d/2} e^{ix \cdot \xi}$ then \eqref{eq:ap:1} and \eqref{eq:ap:2} are satisfied and the claim follows from \eqref{eq:traces}.
\end{proof}

If we choose $\phi \equiv 1$ on $\Omega$ we get
$$
\tr \lk  -\Delta_b - \Lambda \rk_- \geq  L^{(1)}_d |\Omega| \Lambda^{1+d/2}  - \frac{\omega_d}{(2\pi)^d} \int_{\partial \Omega} b(x) d\sigma(x)  \Lambda^{d/2} \, .
$$
This generalizes the bound proved in \cite{Kro92} for the case of Neumann boundary conditions.


\begin{thebibliography}{Wey12}

\bibitem[AS64]{AbrSte64}
M.~Abramowitz and I.~A.~Stegun, \emph{Handbook of mathematical functions
  with formulas, graphs, and mathematical tables}, National Bureau of Standards
  Applied Mathematics Series, vol.~55, 1964.

\bibitem[BS80]{BirSol80}
M.~{\v{S}}. Birman and M.~Z. Solomjak, \emph{Quantitative analysis in {S}obolev
  imbedding theorems and applications to spectral theory}, American
  Mathematical Society Translations, Series 2, vol. 114, American Mathematical
  Society, Providence, R.I., 1980.

\bibitem[BG90]{BrGi90}
T. P. Branson and P. B. Gilkey, \emph{The asymptotics of the {L}aplacian on a manifold with boundary}, Comm. Partial Differential Equations \textbf{15} (1990), no. 2, 245--272.

\bibitem[FG11]{FraGei11}
R.~L. Frank and L.~Geisinger, \emph{Two-term spectral asymptotics of the
  {D}irichlet {L}aplacian on a bounded domain}, Mathematical {R}esults in
  {Q}uantum {P}hysics: {P}roceedings of the {Q}math11 {C}onference (Pavel
  Exner, ed.), World Scientific Publishing Company, 2011, pp.~138--147.

\bibitem[FG12]{FraGei12}
\bysame, \emph{Refined semiclassical asymptotics for fractional powers of the
  {L}aplace operator}, submitted (2012).

\bibitem[H{\"o}r85]{Hoe85}
L.~H{\"o}rmander, \emph{The analysis of linear partial differential operators,
  vol. 4}, Springer-Verlag, Berlin, 1985.

\bibitem[Ivr80a]{Ivr80a}
V.~Ja. Ivrii, \emph{The second term of the spectral asymptotics for the
  {L}aplace-{B}eltrami operator on manifolds with boundary}, Funktsional. Anal.
  i Prilozhen. \textbf{14} (1980), no.~2, 25--34.

\bibitem[Ivr80b]{Ivr80b}
\bysame, \emph{The second term of the spectral asymptotics for the
  {L}aplace-{B}eltrami operator on manifolds with boundary and for elliptic
  operators acting in vector bundles}, Soviet Math. Dokl. \textbf{20} (1980),
  no.~1, 1300--1302.

\bibitem[Ivr98]{Ivr98}
\bysame, \emph{Microlocal analysis and precise spectral asymptotics}, Springer
  Monographs in Mathematics, Springer-Verlag, Berlin, 1998.

\bibitem[Kr{\"o}92]{Kro92}
P.~Kr{\"o}ger, \emph{Upper bounds for the {N}eumann eigenvalues on a bounded
  domain in {E}uclidean space}, J. Funct. Anal. \textbf{106} (1992), no.~2,
  353--357.

\bibitem[MS67]{McKSin67}
H. P. McKean Jr. and I.M. Singer, \emph{Curvature and the eigenvalues of the {L}aplacian}. J. Differential Geometry \textbf{1} (1967), no. 1, 43--69.

\bibitem[Ple54]{Ple54}
\AA. Pleijel, \emph{A study of certain {G}reen's functions with applications in the theory of vibrating membranes}. 
Ark. Mat. \textbf{2} (1954), 553--569.

\bibitem[SS03]{SolSpi03}
J.~P. Solovej and W.~L. Spitzer, \emph{A new coherent states
  approach to semiclassics which gives {S}cott's correction}, Comm. Math. Phys.
  \textbf{241} (2003), no.~2-3, 383--420.

\bibitem[SV97]{SafVas97}
Y.~Safarov and D.~Vassiliev, \emph{The asymptotic distribution of eigenvalues
  of partial differential operators}, Translations of Mathematical Monographs,
  155, American Mathematical Society, Providence, RI, 1997.

\bibitem[Wey12]{Wey12a}
H.~Weyl, \emph{Das asymptotische {V}erteilungsgesetz der {E}igenwerte linearer
  partieller {D}ifferentialgleichungen (mit einer {A}nwendung auf die {T}heorie
  der {H}ohlraumstrahlung)}, Math. Ann. \textbf{71} (1912), no.~4, 441--479.

\end{thebibliography}

\end{document}